\documentclass[reqno,english]{amsart}
\usepackage{amsfonts,amsmath,latexsym,verbatim,amscd,mathrsfs,color,array}
\usepackage[colorlinks=true]{hyperref}
\usepackage{amsmath,amssymb,amsthm,amsfonts,graphicx,color,mathtools}
\usepackage[hmargin=2.3cm, vmargin=2.3cm]{geometry}
\usepackage{bbm}
\usepackage{amssymb}
\usepackage{pdfsync}
\usepackage{epstopdf}
\usepackage{cite}
\usepackage{graphicx}
\usepackage{multirow}
\usepackage[font=bf,aboveskip=15pt]{caption}
\usepackage[toc,page]{appendix}
\usepackage[colorlinks=true]{hyperref}
\newcommand{\pp}{ {\partial} }

\newcommand{\1}{\mathbf{1}}

\newcommand{\RR}{\mathbb{R}}

\newcommand{\la}{\lambda}

\newcommand{\R}{\mathbb R}

\newcommand{\bt}{\begin{theorem}}
\newcommand{\et}{\end{theorem}}
\newcommand{\bl}{\begin{lemma}}
\newcommand{\el}{\end{lemma}}
\newcommand{\bd}{\begin{definition}}
\newcommand{\ed}{\end{definition}}
\newcommand{\bc}{\begin{corollary}}
\newcommand{\ec}{\end{corollary}}
\newcommand{\bp}{\begin{proof}}
\newcommand{\ep}{\end{proof}}
\newcommand{\bx}{\begin{example}}
\newcommand{\ex}{\end{example}}
\newcommand{\bi}{\begin{exercise}}
\newcommand{\ei}{\end{exercise}}
\newcommand{\bo}{\begin{prop}}
\newcommand{\eo}{\end{prop}}
\newcommand{\br}{\begin{remark}}
\newcommand{\er}{\end{remark}}
\newcommand{\be}{\begin{equation}}
\newcommand{\ee}{\end{equation}}
\newcommand{\ba}{\begin{align}}
\newcommand{\ea}{\end{align}}
\newcommand{\bn}{\begin{enumerate}}
\newcommand{\en}{\end{enumerate}}
\newcommand{\bg}{\begin{align*}}
\newcommand{\bcs}{\begin{cases}}
\newcommand{\ecs}{\end{cases}}

\newcommand{\DD}{{\mathcal D}}

\newcommand{\TT}{\mathcal{T}}

\newcommand{\B}{ \mathcal{B} }

\newcommand{\bean}{\begin{eqnarray*}}
\newcommand{\eean}{\end{eqnarray*}}

\newtheorem{definition}{Definition}[section]
\newtheorem{theorem}{Theorem}[section]
\newtheorem{lemma}{Lemma}[section]

\newtheorem{prop}{Proposition}[section]
\newtheorem{remark}{Remark}[section]

\numberwithin{equation}{section}

\begin{document}
\title[Long-time behavior for HMF]{Trichotomy dynamics of the 1-equivariant harmonic map flow}

\author[J. Wei]{Juncheng Wei}
\address{\noindent
Department of Mathematics,
University of British Columbia, Vancouver, B.C., V6T 1Z2, Canada}
\email{jcwei@math.ubc.ca}

\author[Q. Zhang]{Qidi Zhang}
\address{\noindent
Department of Mathematics,
University of British Columbia, Vancouver, B.C., V6T 1Z2, Canada}
\email{qidi@math.ubc.ca}

\author[Y. Zhou]{Yifu Zhou}
\address{\noindent
Department of Mathematics,
Johns Hopkins University, 3400 N. Charles Street, Baltimore, MD 21218, USA}
\email{yzhou173@jhu.edu}

\begin{abstract}
For the 1-equivariant harmonic map flow from $\R^2$ into $S^2$
\begin{equation*}
\left\{
\begin{aligned}
&v_t=v_{rr}+\frac{v_r}{r} - \frac{\sin(2v)}{2r^2} , ~\quad(r,t)\in \R_+\times (t_0,+\infty),\\
&v(r,t_0)=v_0, \qquad\qquad\qquad\quad r\in \R_+,
\end{aligned}
\right.
\end{equation*}
we construct global growing, bounded and decaying solutions with the initial data $v_0(r)$ satisfying
$$v_0(0)=\pi ~\mbox{ and }~ v_0(r)\sim r^{1-\gamma} ~\mbox{ as }~ r\to+\infty, \quad  \gamma>1.$$
These global solutions exhibit the following trichotomy long-time asymptotic behavior
\begin{equation*}
\| v_r(\cdot,t) \|_{L^\infty ([0,\infty))} \sim
\begin{cases}
t^{\frac{\gamma-2}{2}}\ln t  ~&\mbox{ if }~ 1<\gamma<2,\\
1 ~&\mbox{ if }~ \gamma=2,\\
\ln t ~&\mbox{ if }~ \gamma>2,\\
\end{cases}
~\mbox{ as }~ t\to +\infty.
\end{equation*}
\end{abstract}
\maketitle

%%%%%%%%%%%%%%%%%%%%%%%%%%%%%%%%%%%%%%%%%%%%%%%%%%%%%%%%%%

%\tableofcontents
\medskip

\section{Introduction and Main Results}

We consider the harmonic map flow (HMF) from $\R^2$ into $S^2$
\begin{equation*}
\begin{cases}
u_t = \Delta u + |\nabla u|^2 u ~&\mbox{ in }~\R^2\times(0,+\infty),\\
u(\cdot,0)=u_0 ~&\mbox{ in }~\R^2.\\
\end{cases}
\end{equation*}
HMF formally corresponds to the negative $L^2$-gradient flow for the Dirichlet energy
$$
\mathcal E[u]=\int_{\R^2}|\nabla u|^2,
$$
which is decreasing along smooth solutions. A special class of solutions are given by the $k$-equivariant ansatz
$$
u(re^{i\theta},t)=\Big(\cos(k\theta)\sin v,\sin(k\theta)\sin v,\cos v\Big),
$$
and thus HMF gets reduced to a scalar equation for the polar angle
\begin{equation}\label{eqn-HMF}
\left\{
\begin{aligned}
&v_t=v_{rr}+\frac1r v_r -\frac{k^2\sin(2v)}{2r^2}, \quad (r,t)\in \R_+\times\R_+\\
&v(r,0)=v_0, ~\qquad\qquad\qquad\qquad r\in \R_+.\\
\end{aligned}
\right.
\end{equation}

\medskip

In the energy critical dimension $n=2$, the scaling invariance of the Dirichlet energy $\mathcal E[u]$ gives rise to the energy concentration and a natural question of singularity formation versus global regularity. Asymptotic profile decomposition has been studied in seminal works by Struwe \cite{Struwe1985CMH}, Qing \cite{Qing95CAG}, Ding-Tian \cite{DT95CAG}, Wang \cite{Wang96HJM}, Qing-Tian \cite{QT97CPAM} and Topping \cite{Topping04MathZ}. In a recent work \cite{LawrieHMF}, Jendrej and Lawrie proved that the bubble decomposition in the $k$-equivariant class can be in fact taken continuously in time. Finite time blow-up for the two-dimensional HMF has also received much attention since the work by Chang, Ding and Ye \cite{Chang92}. Formal prediction of singularity with quantized blow-up rates was made by van den Berg, Hulshof and King \cite{vandenBerg03}, and this was later rigorously proved by Rapha\"el and Schweyer \cite{Raphael13,Raphael14}. Beyond the equivariant class, multi-bubble blow-up at finite time was constructed recently by D\'avila, del Pino and Wei \cite{17HMF}.

\medskip

HMF is a borderline case of the Landau-Lifshitz-Gilbert equation (LLG)
$$
u_t=a(\Delta u+|\nabla u|^2 u)+b u\wedge \Delta u, \quad a^2+b^2=1,~ a\geq 0, ~b\in\R,
$$
which models the evolution of isotropic ferromagnetic spin fields. A series of works by Gustafson, Kang, Nakanishi and Tsai (in various combinations) \cite{Gustafson07, Gustafson08, Gustafson10}  aimed at investigating the behavior of the solutions to LLG near $k$-equivariant harmonic maps. They found, among other things,
that there is no finite time singularity for LLG with $k\ge 3$ and for HMF with $k=2$ near $k$-equivariant harmonic maps. Intriguingly, in the $2$-equivariant case, they classified  the dynamics of scaling parameter $\mu (t)$ of the map as
\begin{equation}\label{d2dynamics}
\log \mu (t)\sim \frac2\pi \int_1^{\sqrt t}\frac{v_1(s)}{s}ds,
\end{equation}
yielding trichotomy dynamics:asymptotical stability, infinite time blow-up and eternal oscillation, see \cite[Theorem 1.2]{Gustafson10}. Here, $v_1(r)$ is the first entry of the initial degree 2 map. However, the 1-equivariant case is left open, due to the very slow spatial decay of the harmonic map components. The goal of  this paper is to fill this gap.

\medskip

In this paper, we consider the two-dimensional HMF into $S^2$ in the 1-equivariant class
\begin{equation}\label{eqn-scalar}
\left\{
\begin{aligned}
&v_t=v_{rr}+\frac{v_r}{r} - \frac{\sin(2v)}{2r^2} , ~\quad(r,t)\in \R_+\times (t_0,+\infty),\\
&v(r,t_0)=v_0, \qquad\qquad\qquad\quad r\in \R_+,
\end{aligned}
\right.
\end{equation}
where $t_0>0$ is some large initial time. The aim of the paper is to understand possible long-term behavior of \eqref{eqn-scalar}, and the main result stated below depends precisely on the power decay rate of the initial data $v_0$.

\begin{theorem}\label{thm}
For $t_0$ sufficiently large, there exist initial data $v_0(r)$ with $v_0(0)=\pi$ and $v_0(r)\sim r^{1-\gamma}$ as $t\to +\infty$ for any $\gamma>1$ such that the global solutions to \eqref{eqn-scalar} satisfy the following
\begin{equation*}
\| v_r(\cdot,t) \|_{L^\infty ([0,\infty))} \sim
\begin{cases}
t^{\frac{\gamma-2}{2}}\ln t  ~&\mbox{ if }~ 1<\gamma<2,\\
1 ~&\mbox{ if }~ \gamma=2,\\
\ln t ~&\mbox{ if }~ \gamma>2,\\
\end{cases}
~\mbox{ as }~ t\to +\infty.
\end{equation*}
More precisely, the polar angle $v$ takes the form
$$
v(r,t)\sim \eta\left(\frac{r}{\sqrt t}\right)\left[\pi-2\arctan\left(\frac{r}{\mu(t)}\right)\right]
$$
with
\begin{equation*}
\mu(t)\sim
\begin{cases}
t^{\frac{2-\gamma}{2}}(\ln t)^{-1}, \quad &1<\gamma<2,\\
1, \quad &\gamma=2,\\
(\ln t)^{-1}, \quad &\gamma>2.\\
\end{cases}
\end{equation*}
Here $\eta$ is a cut-off function.
\end{theorem}

Our study of the long-time behavior is in fact motivated by a notable connection between the critical Fujita equation in $\R^4$
\begin{equation}
\label{4D}
\begin{cases}
u_t=\Delta u+u^3 ~&\mbox{ in }~\R^4\times \R_+\\
u(\cdot, 0)=u_0 ~&\mbox{ in }~\R^4\\
\end{cases}
\end{equation}
and the HMF with 1-equivariant symmetry. This connection has already been observed in \cite{Schweyer12JFA,Raphael13,17HMF}, and these two equations share similar structure in certain sense. Roughly speaking, these two equations are both energy critical, and the HMF with 1-equivariant symmetry can be viewed as a four-dimensional heat equation in the remote region.  In \cite{FilaKing12}, Fila and King proposed a diagram and conjectured that the long-time asymptotics of threshold solutions to (\ref{4D}) are determined by the power decay rate of the initial data in a rather precise manner. More precisely, they conjectured that for
\begin{equation}\label{Fujita}
	\begin{cases}
		u_t =\Delta u + |u|^{\frac{4}{N-2}}u, &
		x\in \RR^N, \quad t>0,
		\\
		u(x,0)=u_0(x), &
		x\in \RR^N
	\end{cases}
\end{equation}
with initial data
$|u_0|\sim\langle x\rangle^{-\tilde\gamma},$ the $\| u (\cdot,t) \|_{L^\infty(\R^N)}$-norm of threshold solution obeys

\bigskip

\begin{center}
  \begin{tabular}{ | c | c | c | c | }
    \hline
     & $\frac{N-2}{2} <\tilde\gamma <2$ & $\tilde\gamma=2$ & $\tilde\gamma>2$ \\ \hline
    $N=3$ & $t^{\frac{\tilde\gamma-1}{2}}$ & $t^{\frac{1}{2}} (\ln t)^{-1}$ & $t^{\frac{1}{2}}$
		\\     \hline
		$N=4$ & $t^{-\frac{2-\tilde\gamma}{2} } \ln t$ & $1$ & $\ln t$
		\\     \hline
		$N=5$ & $t^{-\frac{3(2-\tilde\gamma)}{2}}$ & $(\ln t)^{-3}$ & $1$ \\
    \hline
  \end{tabular}
\end{center}
\bigskip
In particular, the trichotomy constructed in Theorem \ref{thm} can be viewed as an analogue of the Fila-King diagram in $\R^4$. Recently, global unbounded solutions for Fujita equation \eqref{Fujita} in $\R^3$ and $\R^4$ have been rigorously constructed in \cite{173D,infi4D}, confirming the existence of upper off-diagonal entries in above diagram (including a sub-case $1<\tilde\gamma<2$ when $N=3$). The global decaying solutions in $\R^5$ will be constructed in a forthcoming work \cite{decay5d}.
%\begin{equation}
%\|u(\cdot,t)\|_{L^{\infty}(\R^4)}\sim
%\begin{cases}
%t^{-\frac{2-\tilde\gamma}{2}}\ln t ~&\mbox{ if }~1<\tilde\gamma<2,\\
%1 ~&\mbox{ if }~\tilde\gamma=2,\\
%\ln t ~&\mbox{ if }~\tilde\gamma>2,\\
%\end{cases}
%\end{equation}

\medskip

In the case of the disk with Dirichlet boundary,  the infinite-time bubbling of 1-equivariant HMF and Fujita equation have been studied by Angenent-Hulshof \cite{Angenent-Hulshof} and by Galaktionov-King \cite{King03JDE}, respectively. Their methods and techniques include a careful formal matching of asymptotic expansions and the use of sub- and super-solutions. On the other hand, the global decaying threshold and non-threshold solutions of Fujita equation have been studied extensively, see \cite{Quittner08DCDS,Fila08JMAA,Fila08MA,Gui01JDE,Kavian87AIHP,Kawanago96AIHP,Lee92TAMS,Polacik07Indiana,Polacik03MA,Suzuki1999} as well as a comprehensive book by Quittner and Souplet \cite{Souplet07book} and the references therein. Finally we should also mention some related interesting work on threshold dynamics for energy-critical wave equation by Krieger, Nakanishi and Schlag in \cite{KNS1, KNS2, KNS3}.

\medskip

The method of our construction is different from those used in aforementioned references, and this seems to be the first gluing construction of decaying solutions. In contrast to the local dynamics \eqref{d2dynamics} when $k=2$, the slow spatial decay for the 1-equivariant case in fact triggers a subtle non-local dynamics of the dilation.  The heart of the construction is a non-local dynamics, analogous to (\ref{d2dynamics}), governing the  scaling parameter $\mu(t)$ in a unified way:
\begin{equation}\label{intro-nonlocal}
\underbracket[0.5pt][7pt]{\int_{t/2}^{t-\mu^2(t)}\frac{\dot\mu(s)}{t-s}ds}_{:=I_{{\rm nl}}}+\underbracket[0.5pt][7pt]{\frac{\mu(t)}{t} \vphantom{\int_{t/2}^{t-\mu^2(t)}\frac{\dot\mu(s)}{t-s}ds} }_{:=I_{{\rm ss}}}\sim  \underbracket[0.5pt][7pt]{2 C_{\gamma}
v_{\gamma}(t)  \vphantom{\int_{t/2}^{t-\mu^2(t)}\frac{\dot\mu(s)}{t-s}ds}}_{:=I_{{\rm ic}}} , \quad \forall \gamma>1.
\end{equation}
Here, $I_{{\rm nl}}$ is in fact from a non-local correction dealing with the slow spatial decay.
Such non-local/global feature usually appears in lower dimensional problems and was first observed in \cite{17HMF,173D}. The second term $I_{{\rm ss}}$ comes from a self-similar correction improving the error in the intermediate region, and the last term $I_{{\rm ic}}$ is the contribution from the initial condition $v_0$ whose expression depends only on $\gamma$ (cf. \eqref{def-v_gamma}). The trichotomy in Theorem \ref{thm} is captured by approximating the non-local problem by a leading ODE, but the solvability of the full non-local problem is rather involved.

\medskip

The rest of the paper is devoted to the construction of Theorem \ref{thm}.

\medskip

\noindent \textbf{Notation:} For admissiable functions $g(x), h(x,t)$, denote
\begin{equation*}
\left(T_{n}\circ g \right)(x,t,t_0) := (4\pi t)^{-\frac{n}{2} }\int_{\R^n} e^{-\frac{|x-y|^2}{4t}} g(y) dy,\quad
\left( T_{n}  \bullet g \right) (x,t,t_0) :=
\int_{t_0}^t \int_{\R^n}
e^{ -\frac{|x-y|^2}{4(t-s)}  } h(y,s) dy ds .
\end{equation*}

We write $a\lesssim b$ ($a \gtrsim b$) if there exists a constant $C > 0$ such that $a \le  Cb$ ($a \ge  Cb$) where $C$ is independent of $t$, $t_0$. Set $a \sim b$ if $b \lesssim a \lesssim b$. The Japanese bracket denotes $\langle x\rangle = \sqrt{|x|^2+1}$.

\medskip

\medskip

\section{Approximation and corrections}

\medskip

The first approximation is built on the one parameter family of steady states to the scalar equation \eqref{eqn-HMF} $$
Q_\mu =\pi-2\arctan\left( \frac{r}{\mu} \right), \quad \mu>0.
$$
Then we have
$$
\sin(2Q_\mu)=\frac{4\rho(\rho^2-1)}{(\rho^2+1)^2}, \quad \cos(2Q_\mu)-1=-\frac{8\rho^2}{(\rho^2+1)^2}.
$$
Define the cut-off function $\eta$ as $\eta(r)=1$ for $0\leq r\leq1$ and $\eta(r)=0$ for $r\geq 2$. We take the first approximate solution of the flow \eqref{eqn-HMF} to be
$$
v_*= \eta\left(\frac{r}{\sqrt t}\right)Q_\mu,\quad \mu=\mu(t),
$$
and define the error operator as
$$
E[v]:=-v_t+v_{rr}+\frac1r v_r -\frac{\sin(2v)}{2r^2}.
$$
Let us write
$$\rho:=\frac{r}{\mu},\quad z:=\frac{r}{\sqrt t}.$$
Then we have
\begin{equation}\label{firsterror}
\begin{aligned}
& E[v_*]
=\mu^{-1}\dot\mu\rho \pp_\rho Q_\mu \eta(z)+ \frac1t \eta''(z) Q_\mu +\frac2{\mu\sqrt t}\eta'(z)\pp_{\rho} Q_\mu\\
&~ +\left(\frac{r}{2t\sqrt t}+\frac1{r\sqrt t}\right)\eta' (z) Q_\mu +\eta(z)\frac{\sin(2Q_\mu)}{2r^2}-\frac{\sin(2\eta(z) Q_\mu)}{2r^2}\\
=&~\underbrace{\mu^{-1}\dot\mu \eta(z) \rho \pp_\rho Q_\mu}_{:=\mathcal E_{1}}+\underbrace{\frac2{t\rho}\eta''(z)-\frac4{\mu\sqrt t\rho^2}\eta'(z)+\frac2{\rho}\left(\frac{r}{2t\sqrt t}+\frac1{r\sqrt t}\right)\eta'(z)}_{:=\mathcal E_{21}} \\
&~+ \underbrace{\frac1t \eta''(z) \left(Q_\mu-\frac2{\rho}\right) +\frac2{\mu\sqrt t}\eta'(z)\left(\pp_{\rho} Q_\mu+\frac2{\rho^2}\right)+\left(\frac{r}{2t\sqrt t}+\frac1{r\sqrt t}\right)\eta'(z) \left(Q_\mu-\frac2{\rho}\right)}_{:=\mathcal E_{22}}\\
&~ +\eta(z)\frac{\sin(2Q_\mu)}{2r^2}-\frac{\sin(2\eta(z) Q_\mu)}{2r^2}.
\end{aligned}
\end{equation}
Denote $\mathcal{E}_2:=\mathcal E_{21}+\mathcal E_{22}$. We add two corrections $\Phi_1$ and $\Phi_2$ to transfer the error $\mathcal E_1$, $ \mathcal{E}_2$ of slow spatial decay, where
\begin{equation}\label{eqn-nonlocal}
\pp_t \Phi_1=\pp_{rr} \Phi_1 +\frac1r\pp_r \Phi_1-\frac1{r^2}\Phi_1+\mathcal E_1,
\end{equation}
\begin{equation}\label{eqn-selfsimilar}
\pp_t \Phi_2=\pp_{rr} \Phi_2 +\frac1r\pp_r \Phi_2-\frac1{r^2}\Phi_2+\mathcal E_{2}.
\end{equation}

On the other hand, the contribution from the initial data $v_0$ is also important. Set
\begin{equation*}
	\pp_{t} \Psi_{*}=
	\pp_{rr} \Psi_{*}   + \frac{1}{r} \pp_{r}\Psi_{*}
	-\frac{1}{r^2} \Psi_{*},
	\quad
	\Psi_{*}(r,0) = r\langle r\rangle^{-\gamma}
\end{equation*}
where
\begin{equation*}
	\Psi_{*}(r,t)  =
	r \psi_*(r,t),
	\quad
\psi_*(r,t)=
\left( 4\pi t \right)^{-2}
\int_{\R^{4}}
e^{-\frac{|r \mathbf{e}_1-y|^2}{4 t }} \langle y\rangle^{-\gamma} dy
\end{equation*}
$\mathbf{e}_1=[1,0,0,0]$.
 For $t\ge 1$, by \cite{decay5d}, the leading term from the Cauchy data is given by
\begin{equation*}
	\left( 4\pi t \right)^{-2}
	\int_{\R^{4}}
	e^{-\frac{|y|^2}{4 t }} \langle y\rangle^{-\gamma} dy
	=v_{\gamma}(t) (C_{\gamma} + g_{\gamma}(t))
\end{equation*}
where
\begin{equation}\label{def-v_gamma}
	v_{\gamma}(t)
	=
	\begin{cases}
		t^{-\frac{\gamma}{2}},
		&
		\gamma <4
		\\
		t^{-2} \ln (1+t),
		&
		\gamma =4
		\\
		t^{-2},
		&
		\gamma >4,
	\end{cases}, \quad
	C_{\gamma} =
	\begin{cases}
		(4\pi)^{-2} \int_{\RR^4}
		e^{-\frac{|z|^2}{4 } }
		|z|^{- \gamma } \mathrm{d} z,
		&
		\gamma<4
		\\
		\left(4\pi \right)^{-2} \frac{1}{2} |S^{3}|
		&
		\gamma=4
		\\
		\left(4\pi \right)^{-2} \int_{\RR^4}
		\langle y \rangle^{-\gamma} \mathrm{d}y
		&
		\gamma>4,
	\end{cases}
\end{equation}
\begin{equation*}
	g_{\gamma}(t)
	=
	O\Big(
	\begin{cases}
		t^{-1},   & \gamma <2
		\\
		t^{-1} \langle \ln t \rangle,  &
		\gamma = 2
		\\
		t^{ \frac{\gamma -4 }{2}},  &
		2 < \gamma <4
		\\
		( \ln (1+t)  )^{-1},
		&
		\gamma = 4
		\\
		t^{\frac{4-\gamma}{2}},
		&
		\gamma<6
		\\
		t^{-1}  \langle \ln t \rangle,
		&
		\gamma=6
		\\
		t^{-1}
		&
		\gamma>6
	\end{cases}
	\Big).
\end{equation*}
The remainder term is bounded by
\begin{equation*}
	\begin{aligned}
		&
		\Big|
		\left( 4\pi t \right)^{-2}
		\int_{\R^{4}}
		\left( e^{-\frac{|\mu \rho \mathbf{e}_1-y|^2}{4 t }} - e^{-\frac{|y|^2}{4 t }} \right) \langle y\rangle^{-\gamma} dy  \Big|
		\\
		= \ &
		\Big|
		\left( 4\pi t \right)^{-2}
		\int_{\R^{4}}
		\int_{0}^1 e^{-\frac{|\theta \mu \rho \mathbf{e}_1-y|^2}{4 t }}
		\frac{-(\theta \mu \rho \mathbf{e}_1 -y) \cdot \mu \rho \mathbf{e}_1}{2t}
		\langle y\rangle^{-\gamma}  d\theta dy  \Big|
		\\
		\lesssim \ &
		\mu \rho  t^{-\frac{5}{2}}
		\int_{\R^{4}}
		\int_{0}^1 e^{-\frac{|\theta \mu \rho \mathbf{e}_1-y|^2}{8 t }}
		\langle y\rangle^{-\gamma} d\theta dy
		\lesssim
		\mu \rho  t^{-\frac{1}{2}} v_{\gamma}(t) .
	\end{aligned}
\end{equation*}
Thus we have
\begin{equation}\label{psi*-est}
	\psi_* =  v_{\gamma}(t) (C_{\gamma} + g_{\gamma}(t))
	+ O( \mu \rho  t^{-\frac{1}{2}} v_{\gamma}(t) ) .
\end{equation}

\medskip

As the leading term of $\mu$, $\mu_0$ is written as
\begin{equation}\label{def-mu_0}
	\mu_0(t)  =
	\begin{cases}
		(1-\frac{\gamma}{2})^{-1} (\gamma-1)^{-1} 2C_{\gamma} t^{1-\frac{\gamma}{2}} (\ln t)^{-1},
		&
		1<\gamma<2
		\\
2C_{\gamma} + (\ln t)^{-1} ,
&
\gamma=2
\\
		(\ln t)^{-1} ,
		&
		\gamma>2
	\end{cases}
\end{equation}
and we make the ansatz $\mu(t)\sim \mu_0(t)$, $\dot{\mu}(t)\sim \dot{\mu}_0(t)$ throughout this paper.
The rigorous derivation about the dynamics of $\mu_0$ is given in section \ref{sec-elliptic-improve}.

\medskip

\subsection{Non-local corrections}

Set $\Phi_i=r\varphi_i$, $i=1,2$. Then  for the purpose of finding the solutions of \eqref{eqn-nonlocal} and  \eqref{eqn-selfsimilar}, it suffices to consider
\begin{equation*}
\pp_{t} \varphi_i = \pp_{rr} \varphi_i +
\frac 3{r} \pp_{r} \varphi_i + r^{-1}\mathcal{E}_i .
\end{equation*}
Notice that
\begin{equation*}
r^{-1} \mathcal{E}_1 =  \frac{-2\mu^{-2} \dot{\mu}}{\mu^{-2}r^2+1 } \eta(z) ,
\quad
r^{-1}\mathcal{E}_{21}= 2 \mu t^{-2}
\left( z^{-2} \eta''(z) + 2^{-1}z^{-1}\eta'(z) -z^{-3} \eta'(z) \right),
\quad
r^{-1}\mathcal{E}_{22}
=
O\left( t^{-3} \mu^{3} \1_{\{ \sqrt{t} \le r\le 2  \sqrt{t} \}}\right) .
\end{equation*}
Denote $\varphi=\varphi_1+\varphi_2$. By the same argument for deriving  \cite[Corollary 2.3]{infi4D},
$\varphi_1$ is given by Duhamel's formula
\begin{equation*}
\varphi_{1} =  T_4\bullet (r^{-1}\mathcal E_1)(r,t,\frac{t_0}{2} ) ;
\end{equation*}
the  leading term of $\varphi_2$ is given by the self-similar solution and the rest smaller error is solved by Duhamel's formula.
One making more accurate convolution estimate in the second estimate in p8 \cite{infi4D}, $\varphi$ has the exponential spatial decay  for any fixed time $t$.
The properties of $\varphi$ are described by the following proposition.
\begin{prop}\label{varphi-coro}
	Assume  $\mu_1$ satisfies $|\mu_{1}|\le \frac{\mu}{2}$, $|\dot{\mu}_{1}| \le \frac{|\dot \mu| }{2}$. We have
	\begin{equation*}
		\begin{aligned}
			|\varphi[\mu] |
			\lesssim \ &
			(\mu  t^{-1}
			+ 	g[\mu]  ) \1_{\{ r\le 2t^{\frac 12}  \}}
			+
			\begin{cases}
				|\dot{\mu}| \langle \ln (\mu^{-1} t^{\frac 12} ) \rangle
				&
				\mbox{ \ if \ } r\le \mu
				\\
				|\dot{\mu}| \langle \ln(r^{-1} t^{\frac 12}) \rangle
				&
				\mbox{ \ if \ }  \mu < r\le t^{\frac 12}
				\\
				t |\dot{\mu}| r^{-2} e^{-\frac{r^2}{16 t}}
				&
				\mbox{ \ if \ }   r > t^{\frac 12}
			\end{cases}
			\\
			& +
			O\Big( \mu r^{-2} e^{-\frac{r^2}{16 t}}
			+
|\dot{\mu}| e^{-c_1\frac{r^2}{t}}
			+
			g[\mu] e^{-\frac{r^2}{16 t} }
			\Big) \1_{\{ r > 2t^{\frac 12}  \}}
		\end{aligned}
	\end{equation*}
	where $c_1>0$ is a small constant and
	\begin{equation*}
		g[\mu] = O \Big( 	
		t^{-2} \int_{t_0/2}^{t} ( s^{-1 } \mu^{3}(s)  + 	s |\dot \mu(s)| )  ds
		\Big) .
	\end{equation*}
	
	\begin{equation*}
		\begin{aligned}
			&
			\big|\varphi [\mu+\mu_1] - \varphi [\mu]\big|
			\lesssim
			\left( O ( |\mu_1|  t^{-1} )
			+
			\tilde{g}[\mu,\mu_1] \right) \1_{\{ r  \le 2t^{\frac 12} \}}
			\\
			&
			+
			\sup\limits_{t_1\in[t/2,t]} \Big(\frac{|\mu_1(t_1)|}{ \mu(t) } + \frac{|\dot\mu_{1}(t_1)|}{|\dot\mu(t)|}  \Big)
			\begin{cases}
				|\dot \mu| \langle \ln (\mu^{-1} t^{\frac 12} ) \rangle
				&
				\mbox{ \ if \ } r\le \mu
				\\
				|\dot \mu| \langle\ln(r^{-1} t^{\frac 12}) \rangle
				&
				\mbox{ \ if \ }  \mu < r\le t^{\frac 12}
				\\
				t |\dot \mu| r^{-2} e^{-\frac{r^2}{16 t}}
				&
				\mbox{ \ if \ }   r > t^{\frac 12}
			\end{cases}
			\\
			&
			+
			O \bigg(  \sup\limits_{t_1\in[t/2,t]}|\mu_1(t_1)| r^{-2} e^{-\frac{r^2}{16 t}}
			+
\Big(  \sup\limits_{t_{1}\in [t/2,t] }	|\dot\mu_{1}(t_1)| +  t^{-2}\int_{t_0/2}^t s |\dot\mu_{1}(s)|  ds \Big)
e^{-c_1\frac{r^2}{t}}
			+ \tilde{g}[\mu,\mu_1] e^{-\frac{r^2}{16 t} } \bigg) \1_{\{ r  > 2t^{\frac 12} \}}
		\end{aligned}
	\end{equation*}
	where
	\begin{equation*}
		\begin{aligned}
			&
			\tilde{g}[\mu,\mu_1] =
			O\Big(|\dot\mu  | \ln t
			\sup\limits_{t_{1}\in[t/2,t]} \Big(\frac{|\mu_1(t_1)|}{\mu(t)}
			+ \frac{|\dot\mu_{1}(t_1)|}{|\dot\mu(t)|}\Big)^2\Big)
			\\
			&
			+ 	O \bigg(	  |\dot\mu| \sup\limits_{t_{1}\in[t/2,t]}\Big(\frac{|\mu_1(t_1)|}{ \mu(t) } + \frac{|\dot\mu_{1}(t_1)|}{|\dot\mu(t)|}  \Big)
			+	 t^{-2} \int_{t_0/2}^{t}
			\Big( s^{-1} |\mu_1(s)| \mu^{2}(s) + s |\dot\mu(s)| \Big(\frac{|\mu_1(s)|}{ \mu(s) } + \frac{|\dot\mu_{1}(s)|}{|\dot\mu(s)|}  \Big) \Big)  ds \bigg) .
		\end{aligned}
	\end{equation*}
	More precisely,
	\begin{equation*}%\label{varphi-pre-mu}
		\begin{aligned}
			\varphi[\mu] = \ & \bigg[ -2^{-1} \Big( \mu  t^{-1}
			+ \int_{t/2}^{t-\mu_{0}^2 } \frac{\dot\mu(s)}{t-s} d s
			\Big)
			+ O\big( \mu t^{-2} r^2  +
			|\dot\mu|
			\min\{\frac{ r}{\mu}  ,\ln t \} \big)
			+ g[\mu]
			\bigg]
			\1_{\{ r\le 2t^{\frac 12}  \}}
			\\
			& +
			O\Big( \mu  r^{-2} e^{-\frac{r^2}{16 t}}
			+
			r^{-6} \int_{t_0/2}^t s^2 |\dot\mu(s)|  ds
			+
			g[\mu] e^{-\frac{r^2}{16 t} }
			\Big)
			\1_{\{ r > 2 t^{\frac 12} \}} ,
		\end{aligned}
	\end{equation*}
	\begin{equation*}
		\begin{aligned}
			&
			\varphi[\mu + \mu_1 ] - 	\varphi[\mu]
			=
			\bigg[
			-2^{-1} \Big( \mu_1  t^{-1}
			+
			\int_{t/2}^{t-\mu_{0}^2 } \frac{\dot\mu_{1} (s)}{t-s} d s
			\Big)
			\\
			&
			+ O\Big( |\mu_1| t^{-2} r^2 +  |\dot\mu| \sup\limits_{t_{1}\in[t/2,t]}\Big(\frac{|\mu_1(t_1)|}{ \mu(t)} + \frac{|\dot\mu_{1}(t_1)|}{|\dot\mu(t)|}  \Big)  \frac{ r}{\mu} \Big)
			+
			\tilde{g}[\mu,\mu_1]
			\bigg]
			\1_{\{ r \le 2 t^{\frac 12} \}}
			\\
			&
			+
			O \Big(   \sup\limits_{t_{1}\in [t/2,t] }	|\mu_1(t_1)|  r^{-2} e^{-\frac{r^2}{16 t}}
			+
			r^{-6}
			\Big(
			t^3 \sup\limits_{t_{1}\in [t/2,t] }	|\dot\mu_{1}(t_1)| +  \int_{t_0/2}^{t/2} s^2 |\dot\mu_{1}(s)|  ds \Big) +   e^{-\frac{r^2}{16 t}} 	\tilde{g}[\mu,\mu_1]
			\Big)
			\1_{\{ r > 2 t^{\frac 12} \}}
			.
		\end{aligned}
	\end{equation*}
	
Using the ansatz $\mu(t)\sim \mu_0(t)$, $\dot{\mu}(t)\sim \dot{\mu}_0(t)$, then
\begin{equation}\label{g-est}
	g[\mu]
	\sim
	\begin{cases}
		t^{-\frac{\gamma}{2}} (\ln t)^{-1} , &
		1<\gamma<2
		\\
		t^{-1} (\ln t)^{-2} ,
		& \gamma=2
		\\
		t^{-1} (\ln t)^{-2} ,
		& \gamma>2
	\end{cases}
	\quad
	\sim
	|\dot{\mu}_0| .
\end{equation}

\begin{equation*}
\begin{aligned}
|\varphi[\mu] |
\lesssim
\begin{cases}
\begin{cases}
t^{-\frac{\gamma}{2}}, & r\le \mu_0
\\
t^{-\frac{\gamma}{2}} (\ln t)^{-1}
\langle \ln(r^{-1} t^{\frac{1}{2}}) \rangle,
&
\mu_0 < r\le t^{\frac{1}{2}}
\\
 t^{3-\frac{\gamma}{2}} (\ln t)^{-1} r^{-6} ,
& r>t^{\frac{1}{2}}
\end{cases}
&
\mbox{ \ if \ } 1<\gamma<2
\\
\begin{cases}
t^{-1} , & r\le t^{\frac{1}{2}}
\\
 t^2 r^{-6} ,
& r > t^{\frac{1}{2}}
\end{cases}
&
\mbox{ \ if \ } \gamma = 2
\\
\begin{cases}
(t \ln t)^{-1}
 , & r\le t^{\frac{1}{2}}
\\
t^2 (\ln t)^{-1} r^{-6} ,
& r > t^{\frac{1}{2}}
\end{cases}
&
\mbox{ \ if \ } \gamma > 2 .
\end{cases}
\end{aligned}
\end{equation*}
\end{prop}

\medskip

\section{Further elliptic improvement and the leading dynamics of the $\mu(t)$}\label{sec-elliptic-improve}

\medskip

In order to improve the time decay of the error, we will introduce $\Phi_e$ by solving the linearized elliptic equation. Let us first   denote
\begin{equation*}
v_1(r,t) :=\eta(z) Q_\mu+\Phi_1+\Phi_2+\Psi_{*} +\eta(4z)\Phi_e
\end{equation*}
where $\eta(4z)$ is used to restrict the influence of $\Phi_e $ in the self-similar region.
Then we compute
\begin{equation}
	\begin{aligned}
		E[v_1]
		= \ & -\pp_t \left(\eta(4z) \Phi_e \right)+\pp_{rr} \left(\eta(4z) \Phi_e \right)+\frac1r\pp_r\left(\eta(4z) \Phi_e \right) + \frac{1}{r^2}( \Phi_1+\Phi_2+\Psi_{*})
\\
&
		-\frac{\sin[2(\eta(z) Q_\mu+\Phi_1+\Phi_2 +\Psi_{*}+ \eta(4z) \Phi_e )]}{2r^2}+\eta(z) \frac{\sin(2Q_\mu)}{2r^2}
\\
= \ & -\pp_t \left(\eta(4z) \Phi_e \right)+\pp_{rr} \left(\eta(4z) \Phi_e \right)+\frac1r\pp_r\left(\eta(4z) \Phi_e \right)
-
\eta(4z) \frac{\cos(2Q_\mu)}{r^2}
\Phi_e
\\
&
-
\eta(z) \frac{\cos(2Q_\mu) -1}{r^2}
\left( \Phi_1+\Phi_2+\Psi_{*} \right)
+
E_e
	\end{aligned}
\end{equation}
where
\begin{equation}\label{def-E_e}
\begin{aligned}
E_e
:=&
-
\eta(z) \frac{1}{2r^2}
\Big[ \sin[2( Q_\mu+\Phi_1+\Phi_2+\Psi_{*} + \eta(4z) \Phi_e )]
-
\sin(2Q_\mu)
-
\cos(2Q_\mu) 2\left( \Phi_1+\Phi_2+\Psi_{*} + \eta(4z) \Phi_e \right)
\Big]
\\
&
+
\frac{1}{2r^2}
\Big[
- \sin[2(\eta(z) Q_\mu+\Phi_1+\Phi_2+\Psi_{*} + \eta(4z) \Phi_e )]
+
\eta(z) \sin[2( Q_\mu+\Phi_1+\Phi_2+\Psi_{*} + \eta(4z) \Phi_e )]
\\
& + 2\left( 1-\eta(z)  \right) \left( \Phi_1+\Phi_2+\Psi_{*} \right)
\Big]  .
	\end{aligned}
\end{equation}

Roughly speaking, we will choose $\Phi_e(\rho,t)$ which solves
\begin{equation*}
\pp_{rr} \Phi_e+\frac1r\pp_r\Phi_e-\frac{\cos(2Q_\mu)}{r^2}\Phi_e \approx
\eta(z) \frac{\cos(2Q_\mu)-1}{r^2}( \Phi_1+\Phi_2+ \Psi_{*}),
\end{equation*}
namely
\begin{equation*}
	\pp_{\rho\rho} \Phi_e+\frac1{\rho}\pp_\rho\Phi_e-  \frac{\rho^4-6\rho^2+1}{\rho^2(\rho^2+1)^2}\Phi_e \approx
	\eta(\frac{\mu\rho}{\sqrt{t}})
	\mu  \frac{-8 \rho }{(\rho^2+1)^2}
	\left( \varphi[\mu](\mu\rho,t)
	+
	\psi_*(\mu\rho,t)
	\right).
\end{equation*}
The linearly independent kernels $\mathcal{Z}, \tilde{\mathcal{Z}}$ of
 the homogeneous part satisfying the Wronskian $W[\mathcal{Z}, \tilde{\mathcal{Z} } ] = \rho^{-1}$  are given as follows:
\begin{equation*}
	\mathcal{Z} (\rho)
	= \frac{\rho}{\rho^2 +1} ,
	\quad
	\tilde{\mathcal{Z}}(\rho)
	=
	\frac{\rho^4 +4 \rho^2\ln(\rho) -1}{2 \rho(\rho^2+1)} .
\end{equation*}

Let us write the orthogonality
\begin{equation*}
\begin{aligned}
	\mathcal{M}[\mu]
= \ &
\int_0^\infty
\eta(\frac{\mu\rho}{\sqrt{t}})
\frac{8 \rho }{(\rho^2+1)^2}
\left( \varphi[\mu](\mu\rho,t)
+
\psi_*(\mu\rho,t)
\right) \mathcal{Z} (\rho) \rho d\rho
\\
= \ &
\int_0^\infty \eta(\frac{\mu\rho}{\sqrt{t}}) \frac{8 \rho^3 }{(\rho^2+1)^3}
\left( \varphi[\mu](\mu\rho,t)
+
\psi_*(\mu\rho,t)
\right)  d\rho .
\end{aligned}
\end{equation*}
By \eqref{psi*-est} and Proposition \ref{varphi-coro}, we have
\begin{equation*}
\begin{aligned}
\mathcal{M}[\mu] =  &
\int_0^\infty \eta(\frac{\mu\rho}{\sqrt{t}}) \frac{8 \rho^3 }{(\rho^2+1)^3}
\bigg[ -2^{-1} \Big( \mu  t^{-1}
+ \int_{t/2}^{t-\mu_{0}^2 } \frac{\dot\mu(s)}{t-s} d s
\Big)
\\
&
\qquad+ O\big( \mu^3 t^{-2} \rho^2  +
|\dot\mu|
\min\{\rho ,\ln t \} \big)
+ g[\mu]
\bigg]     d\rho
\\
& +
\int_0^\infty
\eta(\frac{\mu\rho}{\sqrt{t}})
 \frac{8 \rho^3 }{(\rho^2+1)^3}
\left( v_{\gamma}(t) (C_{\gamma} + g_{\gamma}(t))
+ O( \mu \rho  t^{-\frac{1}{2}} v_{\gamma}(t) ) \right) d\rho
\\
= \ &    \Big( \mu  t^{-1}
+ \int_{t/2}^{t-\mu_{0}^2 } \frac{\dot\mu(s)}{t-s} d s
\Big)
\left( -1 + O((t^{\frac{1}{2}} \mu^{-1})^{-2} )\right)
+ O\big( \mu^3 t^{-2} \ln(t^{\frac{1}{2}} \mu^{-1} )  +
|\dot\mu|   \big)
+ g[\mu]
\\
& +
v_{\gamma}(t) (C_{\gamma} + g_{\gamma}(t))
\left( 2 + O((t^{\frac{1}{2}} \mu^{-1})^{-2} )\right)
+
O( \mu  t^{-\frac{1}{2}} v_{\gamma}(t) ) .
\end{aligned}
\end{equation*}
Singling out the leading terms, we then have
\begin{equation}\label{eqn-323232}
 \mu  t^{-1}
+ \int_{t/2}^{t-\mu_{0}^2 } \frac{\dot\mu(s)}{t-s} d s  \approx
2 C_{\gamma}
v_{\gamma}(t)  .
\end{equation}
 Based on this, we now derive the leading term $\mu_0$ of the scaling parameter $\mu$. For $\mu_0(t)$ with the form $ \mu_0(t)  =c_1 t^{1-p_0} (\ln t)^{-1}$ with $p_0<1$, we have $$\dot{\mu}_0(t) = c_1 (1-p_0) t^{-p_0} (\ln t)^{-1}
\left[ 1- (1-p_0)^{-1} (\ln t)^{-1}\right].$$ For $t_1\le \frac{t}{2}$, one has
\begin{equation*}
	\begin{aligned}
		&
		\int_{t_1}^{t-\mu_0^2(t)} \frac{\dot {\mu}_0(s) }{t-s}   ds
		=
		\int_{\frac{t_1}{t}}^{1-\frac{\mu_0^2(t)}{t}} \frac{\dot {\mu}_0(tz) }{1-z}   dz
		\\
		= \ &
		c_1 (1-p_0) t^{-p_0}
		\int_{\frac{t_1}{t}}^{1-\frac{\mu_0^2(t)}{t}} (1-z)^{-1}  z^{-p_0} (\ln (tz))^{-1}
		\left[ 1- (1-p_0)^{-1} (\ln (tz) )^{-1}\right]  dz
		\\
		= \ &
		c_1 (1-p_0) (2p_0-1) t^{-p_0}  +
		O( t^{-p_0} (\ln t)^{-1} \ln( \ln t ) ),
	\end{aligned}
\end{equation*}
where we have used the following estimates in the last step
\begin{equation*}
	\begin{aligned}
		&
		\int_{\frac{t_1}{t}}^{1-\frac{\mu_0^2(t)}{t}} (1-z)^{-1}  z^{-p_0} (\ln (tz))^{-1} dz
		\\
		= \ &
		(\ln t)^{-1}
		\int_{\frac{t_1}{t}}^{1-\frac{\mu_0^2(t)}{t}} (1-z)^{-1}  z^{-p_0}  dz
		+
		\int_{\frac{t_1}{t}}^{1-\frac{\mu_0^2(t)}{t}} (1-z)^{-1}  z^{-p_0} \left( (\ln (tz))^{-1} - (\ln t)^{-1} \right) dz
		\\
		= \ &
		(\ln t)^{-1}
		\int_{\frac{t_1}{t}}^{1-\frac{\mu_0^2(t)}{t}} (1-z)^{-1}  dz
		+
		(\ln t)^{-1}
		\int_{\frac{t_1}{t}}^{1-\frac{\mu_0^2(t)}{t}} (1-z)^{-1} \left( z^{-p_0} -1 \right) dz
		\\
		&
		+
		\int_{\frac{t_1}{t}}^{1-\frac{\mu_0^2(t)}{t}} (1-z)^{-1}  z^{-p_0}
		\frac{-\ln z}{(\ln t+\ln z)\ln t} dz
		\\
		= \ &
		(\ln t)^{-1} \left( -\ln(t^{-1} \mu_0^2(t)) +\ln(1-\frac{t_1}{t}) \right)
		+
		O((\ln t)^{-1})
		\\
		= \ &
		(\ln t)^{-1} \left( -\ln( c_1^2 t^{1-2p_0} (\ln t)^{-2} ) +\ln(1-\frac{t_1}{t}) \right)
		+
		O((\ln t)^{-1})
		\\
		= \ &
		(\ln t)^{-1} \left( -\ln( c_1^2)
		-(1-2p_0)\ln t +2\ln( \ln t )
		+\ln(1-t_1 t^{-1} ) \right)
		+
		O((\ln t)^{-1})
		\\
		= \ &
		2p_0-1 + O((\ln t)^{-1} \ln( \ln t ) ) ;
	\end{aligned}
\end{equation*}
and
\begin{equation*}
	\int_{\frac{t_1}{t}}^{1-\frac{\mu_0^2(t)}{t}} (1-z)^{-1}  z^{-p_0} (\ln (tz))^{-2} dz = O( (\ln t)^{-1} )
\end{equation*}
since
\begin{equation*}
	\begin{aligned}
		&
		\int_{\frac{1}{2}}^{1-\frac{\mu_0^2(t)}{t}} (1-z)^{-1}  z^{-p_0} (\ln (tz))^{-2} dz
		= O( (\ln t)^{-2} ) \int_{\frac{1}{2}}^{1-\frac{\mu_0^2(t)}{t}} (1-z)^{-1} dz
		= O( (\ln t)^{-1} ) ,
		\\
		&
		\int_{\frac{t_1}{t}}^{\frac{1}{2}} (1-z)^{-1}  z^{-p_0} (\ln (tz))^{-2} dz
		\sim
		\int_{\frac{t_1}{t}}^{\frac{1}{2}}  z^{-p_0} (\ln (tz))^{-2} dz
		=
		t^{p_0-1} \int_{t_1}^{\frac{t}{2}} a^{-p_0} (\ln a)^{-2} da
		= O( (\ln t)^{-2} ) .
	\end{aligned}
\end{equation*}

\noindent $\bullet$ For $1<\gamma<2$, in order to balance out
\begin{equation*}
	c_1 t^{-p_0} (\ln t)^{-1}
	+
	c_1 (1-p_0) (2p_0-1) t^{-p_0}  +
	O\left( t^{-p_0} (\ln t)^{-1} \ln( \ln t ) \right)
	\approx
	2 C_{\gamma}
	v_{\gamma}(t)  ,
\end{equation*}
we take
\begin{equation*}
	p_0 = \frac{\gamma}{2},
	\quad
	c_1 = (1-\frac{\gamma}{2})^{-1} (\gamma-1)^{-1} 2C_{\gamma} .
\end{equation*}
This then implies
\begin{equation*}
\mu_0(t)  = (1-\frac{\gamma}{2})^{-1} (\gamma-1)^{-1} 2C_{\gamma} t^{1-\frac{\gamma}{2}} (\ln t)^{-1},
\end{equation*}
and
\begin{equation*}
	 -\Big( \mu_0  t^{-1}
+ \int_{t/2}^{t-\mu_{0}^2 } \frac{\dot\mu_0(s)}{t-s} d s
\Big)
 +
2v_{\gamma}(t) (C_{\gamma} + g_{\gamma}(t))
=
O\left( t^{-\frac{\gamma}{2}} (\ln t)^{-1} \ln \ln t \right)
=
O\left(  \ln \ln t |\dot{\mu}_0| \right)
 .
\end{equation*}

\noindent $\bullet$ For $\gamma=2$, in order to balance out
\begin{equation*}
	\mu  t^{-1}
	+ \int_{t/2}^{t-\mu_{0}^2 } \frac{\dot\mu(s)}{t-s} d s  \approx
	2 v_{\gamma}(t) (C_{\gamma} + g_{\gamma}(t)) ,
\end{equation*}
we choose
\begin{equation*}
	\mu_0 = 2C_{\gamma} + (\ln t)^{-1} .
\end{equation*}
Then
\begin{equation*}
	-\Big( \mu_0  t^{-1}
	+ \int_{t/2}^{t-\mu_{0}^2 } \frac{\dot\mu_0(s)}{t-s} d s
	\Big)
	+
	2v_{\gamma}(t) (C_{\gamma} + g_{\gamma}(t))
	=
	O\left( t^{-1} (\ln t)^{-2} \ln \ln t  \right)
=
O\left( \ln \ln t |\dot{\mu}_0| \right)
	 .
\end{equation*}

\noindent $\bullet$ For $\gamma>2$, by the same argument in \cite[section 2.3]{infi4D}, a good approximation is
\begin{equation*}
\mu_0=(\ln t)^{-1}
\end{equation*}
and
\begin{equation*}
	-\Big( \mu_0  t^{-1}
	+ \int_{t/2}^{t-\mu_{0}^2 } \frac{\dot\mu_0(s)}{t-s} d s
	\Big)
	+
	2v_{\gamma}(t) (C_{\gamma} + g_{\gamma}(t))
	=
	O\left( t^{-1} (\ln t)^{-2} \ln \ln t  \right)
	=
	O\left( \ln \ln t |\dot{\mu}_0| \right)
	 .
\end{equation*}

In conclusion, the non-local problem \eqref{eqn-323232} has a good approximation of the form
\begin{equation}\label{def-mu_0}
	\mu_0(t)  =
	\begin{cases}
		(1-\frac{\gamma}{2})^{-1} (\gamma-1)^{-1} 2C_{\gamma} t^{1-\frac{\gamma}{2}} (\ln t)^{-1},
		&
		1<\gamma<2
		\\
2C_{\gamma} + (\ln t)^{-1} ,
&
\gamma=2
\\
		(\ln t)^{-1} ,
		&
		\gamma>2,
	\end{cases}
\end{equation}
where the constant $C_{\gamma}$ is defined in \eqref{def-v_gamma}.

By the same argument in \cite[section 2.3]{infi4D}, we are able to perform several iterations to find $\bar{\mu}_0$ satisfying
$\bar{\mu}_0\sim \mu_0$ and $\dot{\bar{\mu}}_0 \sim \dot{\mu}_0$ such that
\begin{equation*}
\mathcal{M}[\bar{\mu}_0] = O(t^{-2}) .
\end{equation*}

Combining Proposition \ref{varphi-coro}, for $\mu=\bar{\mu}_0+\mu_1$, with $|\mu_1|\le \frac{\bar{\mu}_0}{2}$, $|\dot{\mu}_1|\le \frac{|\dot{\bar{\mu}}_0|}{2}$, we have
\begin{equation}\label{varphi+psi}
	\begin{aligned}
		&	\varphi[\mu] + \psi_* =   \bigg[ -2^{-1} \Big( \mu  t^{-1}
		+ \int_{t/2}^{t-\mu_{0}^2 } \frac{\dot\mu(s)}{t-s} d s
		\Big)
		+ O\big( \mu t^{-2} r^2  +
		|\dot\mu|
		\min\{\frac{ r}{\mu}  ,\ln t \} \big)
		+ g[\mu]
		\bigg]
		\1_{\{ r\le 2t^{\frac 12}  \}}
		\\
		& +
		O\Big( \mu  r^{-2} e^{-\frac{r^2}{16 t}}
		+
		r^{-6} \int_{t_0/2}^t s^2 |\dot\mu(s)|  ds
		+
		g[\mu] e^{-\frac{r^2}{16 t} }
		\Big)
		\1_{\{ r > 2 t^{\frac 12} \}}
		+
		v_{\gamma}(t) (C_{\gamma} + g_{\gamma}(t))
		+ O( \mu \rho  t^{-\frac{1}{2}} v_{\gamma}(t) )
		\\
		= \ &
		\bigg[ -2^{-1} \Big( \mu_1  t^{-1}
		+ \int_{t/2}^{t-\mu_{0}^2 } \frac{\dot\mu_1(s)}{t-s} d s
		\Big)
		+ O( \mu_0  t^{-2} r^2 )
		+
		|\dot{\mu}_0|
		\min\{ \langle \rho \rangle,\ln t \} \big)
		\bigg]
		\1_{\{ r\le 2t^{\frac 12}  \}}
		\\
		& +
		O\Big(  \mu_0  r^{-2} e^{-\frac{r^2}{16 t}}
		+
		|\dot{\mu}_0|  t^3 r^{-6}
		\Big)
		\1_{\{ r > 2 t^{\frac 12} \}}
		+ O( \mu_0 \rho  t^{-\frac{1}{2}} v_{\gamma}(t) )
		+
		O\left( \ln \ln t |\dot{\mu}_0| \right),
	\end{aligned}
\end{equation}
where we have used \eqref{g-est}.

Since $\bar{\mu}_0$ is determined, we are now able to describe $\Phi_e$ rigorously for the computations of new error later.
Set $\bar{\rho}=\frac{r}{\bar{\mu}_0}$ and consider $\Phi_e=\Phi_e(\bar{\rho},t)$ solving
\begin{equation*}
	\pp_{\bar{\rho}\bar{\rho} } \Phi_e+\frac1{\bar{\rho}}\pp_{\bar{\rho}}\Phi_e-  \frac{\bar{\rho}^4-6\bar{\rho}^2+1}{\bar{\rho}^2(\bar{\rho}^2+1)^2}\Phi_e
	=
	\tilde{H}(\bar{\rho},t)
\end{equation*}
where
\begin{equation*}
	\tilde{H}(\bar{\rho},t)=
	\bar{\mu}_0
	\eta(\frac{\bar{\mu}_0 \bar{\rho}}{\sqrt{t}})
	\frac{-8 \bar\rho }{(\bar{\rho}^2+1)^2}
	\left( \varphi[\bar{\mu}_0](\bar{\mu}_0\bar{\rho},t)
	+
	\psi_*(\bar{\mu}_0 \bar{\rho},t)
	\right)
	+
	\bar{\mu}_0 \mathcal{M}[\bar{\mu}_0]
	\frac{\eta(\bar{\rho}) \mathcal{Z}(\bar{\rho})  }{\int_{0}^3 \eta(x) \mathcal{Z}^2(x) xdx } .
\end{equation*}
$\Phi_e$ is taken as
\begin{equation*}
	\Phi_e(\bar{\rho},t) =
	\tilde{\mathcal{Z}}(\bar{\rho}) \int_{0}^{\bar{\rho}} \tilde{H}(x,t) \mathcal{Z}(x) x dx
	-
	\mathcal{Z}(\bar{\rho}) \int_{0}^{\bar{\rho}} \tilde{H}(x,t) \tilde{\mathcal{Z}}(x) x dx
	.
\end{equation*}

By the definition of $\mathcal{M}[\bar{\mu}_0]$, one clearly has
\begin{equation}\label{tilde-H-orth}
	\int_{0}^\infty \tilde{H}(x, t) \mathcal{Z}(x) xdx =0.
\end{equation}
By Proposition \ref{varphi-coro},  for $1<\gamma<2$,
\begin{equation*}
	\begin{aligned}
		\tilde{H}(\bar{\rho},t)=&~
		\bar{\mu}_0
		\eta(\frac{\bar{\mu}_0 \bar{\rho}}{\sqrt{t}})
		\frac{-8 \bar\rho }{(\bar{\rho}^2+1)^2}
		\Bigg\{
		\bigg[ -2^{-1} \Big( \bar{\mu}_0  t^{-1}
		+ \int_{t/2}^{t-\mu_{0}^2 } \frac{\dot{\bar\mu}_0(s)}{t-s} d s
		\Big)\\
		&
		\qquad\qquad+ O\big( \bar{\mu}_0 t^{-2} (\bar{\mu}_0 \bar{\rho})^2  +
		|\dot{\bar{\mu}}_0|
		\min\{ \bar{\rho}  ,\ln t \} \big)
		+ g[\bar{\mu}_0]
		\bigg]
		\\
		& \qquad\qquad+
		v_{\gamma}(t) (C_{\gamma} + g_{\gamma}(t))
		+ O( \bar{\mu}_0 \bar{\rho}  t^{-\frac{1}{2}} v_{\gamma}(t) )
		\Bigg\}
		+
		\bar{\mu}_0 O(t^{-2})
		\frac{\eta(\bar{\rho}) \mathcal{Z}(\bar{\rho}) \bar{\rho} }{\int_{0}^3 \eta(x) \mathcal{Z}^2(x) xdx }
		\\
		= \ &
		\bar{\mu}_0
		\eta(\frac{\bar{\mu}_0 \bar{\rho}}{\sqrt{t}})
		\frac{-8 \bar\rho }{(\bar{\rho}^2+1)^2}
		\left(
		O( t^{-\frac{\gamma}{2}} (\ln t)^{-1} \ln( \ln t ) )
		+ O\big(
		t^{-\frac{\gamma}{2}} (\ln t)^{-1}
		\bar{\rho} \big)
		\right)
		+
		\bar{\mu}_0 O(t^{-2})
		\frac{\eta(\bar{\rho}) \mathcal{Z}(\bar{\rho})  }{\int_{0}^3 \eta(x) \mathcal{Z}^2(x) xdx }   .
	\end{aligned}
\end{equation*}
Thus
\begin{equation*}
| \tilde{H} |
\lesssim
\mu_0
\eta(\frac{\bar{\mu}_0 \bar{\rho}}{\sqrt{t}})
\bar\rho \langle \bar\rho \rangle^{-3}
 t^{-\frac{\gamma}{2}} (\ln t)^{-1} \ln( \ln t )
 .
\end{equation*}
%$\dfinal$ where we used $1<\gamma$.
Similarly,
for $\gamma=2$, we have
\begin{equation*}
\begin{aligned}
		|\tilde{H} |
		= \ &
		\Big|
		\bar{\mu}_0
		\eta(\frac{\bar{\mu}_0 \bar{\rho}}{\sqrt{t}})
		\frac{-8 \bar\rho }{(\bar{\rho}^2+1)^2}
		\left( 	O\left( t^{-1} (\ln t)^{-2} \ln( \ln t ) \right)
		+
		O\left( t^{-1}(\ln t)^{-2} \bar{\rho}\right)
		\right)
		+
		\bar{\mu}_0 O(t^{-2})
		\frac{\eta(\bar{\rho}) \mathcal{Z}(\bar{\rho})  }{\int_{0}^3 \eta(x) \mathcal{Z}^2(x) xdx }   \Big|
\\
\lesssim \ &
\mu_0
\eta(\frac{\bar{\mu}_0 \bar{\rho}}{\sqrt{t}})
\bar\rho
\langle \bar\rho \rangle^{-3}
 t^{-1} (\ln t)^{-2} \ln( \ln t ) .
\end{aligned}
\end{equation*}
For $\gamma>2$, we have
\begin{equation*}
	\begin{aligned}
		|\tilde{H} |
		= \ &
		\Big|
		\bar{\mu}_0
		\eta(\frac{\bar{\mu}_0 \bar{\rho}}{\sqrt{t}})
		\frac{-8 \bar\rho }{(\bar{\rho}^2+1)^2}
		\left(
		O\left( t^{-1} (\ln t)^{-2} \ln( \ln t ) \right)
		+
		O\left(t^{-1}(\ln t)^{-2} \bar{\rho} \right)
		\right)
		+
		\bar{\mu}_0 O(t^{-2})
		\frac{\eta(\bar{\rho}) \mathcal{Z}(\bar{\rho})  }{\int_{0}^3 \eta(x) \mathcal{Z}^2(x) xdx }   \Big|
\\
\lesssim \ &
\mu_0
\eta(\frac{\bar{\mu}_0 \bar{\rho}}{\sqrt{t}})
\bar\rho
\langle \bar\rho \rangle^{-3}
 t^{-1} (\ln t)^{-2} \ln( \ln t ) .
	\end{aligned}
\end{equation*}
Using the rough estimates in Proposition \ref{varphi-coro}, another upper bound of $\tilde{H}$ is given by
\begin{equation*}
	|\tilde{H}| \lesssim
\mu_0
\eta(\frac{\bar{\mu}_0 \bar{\rho}}{\sqrt{t}})
	\begin{cases}
	\bar\rho \langle \bar\rho  \rangle^{-4}
	t^{-\frac{\gamma}{2}} ,
&
1<\gamma<2
\\
\bar\rho \langle \bar\rho  \rangle^{-4}
t^{-1},
& \gamma=2
\\
\bar\rho \langle \bar\rho  \rangle^{-4}
t^{-1} (\ln t)^{-1},
&
\gamma>2 .
	\end{cases}
\end{equation*}
Combining the above two upper bounds, $\tilde{H}$ is then bounded by
\begin{equation*}
	|\tilde{H}| \lesssim
	\mu_0
	\eta(\frac{\bar{\mu}_0 \bar{\rho}}{\sqrt{t}})
	\begin{cases}
		\min\left\{
		\bar\rho \langle \bar\rho  \rangle^{-4}
		t^{-\frac{\gamma}{2}} ,
		\bar\rho \langle \bar\rho \rangle^{-3}
		t^{-\frac{\gamma}{2}} (\ln t)^{-1} \ln( \ln t )
		\right\}
		,
		&
		1<\gamma<2
		\\
		\min\left\{
		\bar\rho \langle \bar\rho  \rangle^{-4}
		t^{-1}, \bar\rho
		\langle \bar\rho \rangle^{-3}
		t^{-1} (\ln t)^{-2} \ln( \ln t )  \right\} ,
		& \gamma=2
		\\
		\min\left\{
		\bar\rho \langle \bar\rho  \rangle^{-4}
		t^{-1} (\ln t)^{-1}, \bar\rho
		\langle \bar\rho \rangle^{-3}
		t^{-1} (\ln t)^{-2} \ln( \ln t ) \right\},
		&
		\gamma>2 .
	\end{cases}
\end{equation*}
Using \eqref{tilde-H-orth}, we have
\begin{equation}\label{Phie-est1}
\langle \bar\rho  \rangle |\pp_{\bar\rho }\Phi_e|	+ |\Phi_e| \lesssim
	\mu_0 \bar{\rho}^3 \langle \bar{\rho} \rangle^{-3}
	\begin{cases}
		\min\left\{
		t^{-\frac{\gamma}{2}}
	\langle \bar\rho  \rangle^{-1}
	\ln (\bar\rho+2)
		,
		t^{-\frac{\gamma}{2}} (\ln t)^{-1} \ln \ln t
		\right\}
		,
		&
		1<\gamma<2
		\\
		\min\left\{
		t^{-1} \langle \bar\rho  \rangle^{-1}
		\ln (\bar\rho+2),
		t^{-1} (\ln t)^{-2} \ln \ln t   \right\} ,
		& \gamma=2
		\\
		\min\left\{
		t^{-1} (\ln t)^{-1} \langle \bar\rho  \rangle^{-1}
		\ln (\bar\rho+2) ,
		t^{-1} (\ln t)^{-2} \ln \ln t  \right\},
		&
		\gamma>2 .
	\end{cases}
\end{equation}

By the same argument in \cite[(2.28)]{infi4D}, we also have
\begin{equation*}
 |\pp_{t}\Phi_e| \lesssim
	\mu_0  \bar{\rho}^3 \langle \bar{\rho} \rangle^{-3}
	\begin{cases}
		t^{-1-\frac{\gamma}{2}} \ln t
		\langle \bar\rho  \rangle^{-1}
		\ln (\bar\rho+2)
		,
		&
		1<\gamma<2
		\\
		t^{-2} \ln t \langle \bar\rho  \rangle^{-1}
		\ln (\bar\rho+2)  ,
		& \gamma=2
		\\
		t^{-2}  \langle \bar\rho  \rangle^{-1}
		\ln (\bar\rho+2) ,
		&
		\gamma>2 .
	\end{cases}
\end{equation*}

\medskip

\subsection{New error}

\medskip

We now use the expression of $\Phi_e$ to compute the new error.
\begingroup
\allowdisplaybreaks
	\begin{align*}
		&
		E\left[v_1 \right]
		\\
= \ &
\Phi_e \eta'(4z)\frac{2r}{t^{\frac{3}{2}}}-\eta(4z)\pp_t \Phi_e+\frac{16}{t}\eta''(4z)\Phi_e+\frac{8}{\sqrt t}\eta'(4z)\pp_r \Phi_e
+  \frac1r \frac4{\sqrt t}\eta'(4z)\Phi_e
\\
&
+\eta(4z)\pp_{rr}\Phi_e
+\eta(4z) \frac1r \pp_r \Phi_e
-
\eta(4z) \frac{\cos(2Q_\mu)}{r^2}
\Phi_e
-
\eta(z) \frac{\cos(2Q_\mu) -1}{r^2}
\left( \Phi_1+\Phi_2+\Psi_{*} \right)
\\
&
-
\eta(z) \frac{1}{2r^2}
\left[ \sin[2( Q_\mu+\Phi_1+\Phi_2+\Psi_{*} + \eta(4z) \Phi_e )]
-
\sin(2Q_\mu)
-
\cos(2Q_\mu) 2\left( \Phi_1+\Phi_2+\Psi_{*} + \eta(4z) \Phi_e \right)
\right]
\\
&
+
\frac{1}{2r^2}
\Big[
- \sin[2(\eta(z) Q_\mu+\Phi_1+\Phi_2+\Psi_{*} + \eta(4z) \Phi_e )]
+
\eta(z) \sin[2( Q_\mu+\Phi_1+\Phi_2+\Psi_{*} + \eta(4z) \Phi_e )]
\\
& + 2\left( 1-\eta(z)  \right) \left( \Phi_1+\Phi_2+\Psi_{*} \right)
\Big]
\\
= \ &
\Phi_e \eta'(4z)\frac{2r}{t^{\frac{3}{2}}}-\eta(4z)\pp_t \Phi_e+\frac{16}{t}\eta''(4z)\Phi_e+\frac{8}{\sqrt t}\eta'(4z)\pp_r \Phi_e
+  \frac1r \frac4{\sqrt t}\eta'(4z)\Phi_e
\\
&
+\eta(4z) \bar{\mu}_0^{-2}
\left(
\pp_{\bar{\rho} \bar{\rho}}\Phi_e
+  \frac1{\bar{\rho}} \pp_{\bar{\rho}} \Phi_e
-
  \frac{\cos(2Q_{\bar{\mu}_0})}{\bar{\rho}^2}
\Phi_e
\right)
+\eta(4z) \bar{\mu}_0^{-2}
\left(
-
\frac{\cos(2Q_\mu)-\cos(2Q_{\bar{\mu}_0})}{\bar{\rho}^2}
\Phi_e
\right)
\\
&
-
\eta(z) \frac{\cos(2Q_\mu) -1}{r^2}
\left( \Phi_1+\Phi_2+\Psi_{*} \right)
\\
&
-
\eta(z) \frac{1}{2r^2}
\left[ \sin[2( Q_\mu+\Phi_1+\Phi_2+\Psi_{*} + \eta(4z) \Phi_e )]
-
\sin(2Q_\mu)
-
\cos(2Q_\mu) 2\left( \Phi_1+\Phi_2+\Psi_{*} + \eta(4z) \Phi_e \right)
\right]
\\
&
+
\frac{1}{2r^2}
\Big[
- \sin[2(\eta(z) Q_\mu+\Phi_1+\Phi_2+\Psi_{*} + \eta(4z) \Phi_e )]
+
\eta(z) \sin[2( Q_\mu+\Phi_1+\Phi_2+\Psi_{*} + \eta(4z) \Phi_e )]
\\
& + 2\left( 1-\eta(z)  \right) \left( \Phi_1+\Phi_2+\Psi_{*} \right)
\Big] \\
= \ &
\Phi_e \eta'(4z)\frac{2r}{t^{\frac{3}{2}}}-\eta(4z)\pp_t \Phi_e+\frac{16}{t}\eta''(4z)\Phi_e+\frac{8}{\sqrt t}\eta'(4z)\pp_r \Phi_e
+  \frac1r \frac4{\sqrt t}\eta'(4z)\Phi_e
\\
&
+\eta(4z) \bar{\mu}_0^{-1}
\Bigg[
\eta(z)
\frac{-8 \bar\rho }{(\bar{\rho}^2+1)^2}
\left( \varphi[\bar{\mu}_0](r,t)
+
\psi_*(r,t)
\right)
+
  \mathcal{M}[\bar{\mu}_0]
\frac{\eta(\bar{\rho}) \mathcal{Z}(\bar{\rho})  }{\int_{0}^3 \eta(x) \mathcal{Z}^2(x) xdx }
\Bigg]
\\
&
+
\eta(z) \mu^{-1}  \frac{8 \rho }{ \left( \rho^{2}+1\right)^2}
\left( \varphi[\mu] +\psi_{*} \right)
+\eta(4z) \bar{\mu}_0^{-2}
\left(
-
\frac{\cos(2Q_\mu)-\cos(2Q_{\bar{\mu}_0})}{\bar{\rho}^2}
\Phi_e
\right)
\\
&
-
\eta(z) \frac{1}{2r^2}
\left[ \sin[2( Q_\mu+\Phi_1+\Phi_2+\Psi_{*} + \eta(4z) \Phi_e )]
-
\sin(2Q_\mu)
-
\cos(2Q_\mu) 2\left( \Phi_1+\Phi_2+\Psi_{*} + \eta(4z) \Phi_e \right)
\right]
\\
&
+
\frac{1}{2r^2}
\Big[
- \sin[2(\eta(z) Q_\mu+\Phi_1+\Phi_2+\Psi_{*} + \eta(4z) \Phi_e )]
+
\eta(z) \sin[2( Q_\mu+\Phi_1+\Phi_2+\Psi_{*} + \eta(4z) \Phi_e )]
\\
& + 2\left( 1-\eta(z)  \right) \left( \Phi_1+\Phi_2+\Psi_{*} \right)
\Big]
\\
= \ &
\Phi_e \eta'(4z)\frac{2r}{t^{\frac{3}{2}}}-\eta(4z)\pp_t \Phi_e+\frac{16}{t}\eta''(4z)\Phi_e+\frac{8}{\sqrt t}\eta'(4z)\pp_r \Phi_e
+  \frac1r \frac4{\sqrt t}\eta'(4z)\Phi_e
\\
&
+
\eta(4z) \bar{\mu}_0^{-1}
\frac{-8 \bar\rho }{(\bar{\rho}^2+1)^2}
\left( \varphi[\bar{\mu}_0](r,t)
+
\psi_*(r,t)
\right)
+
\eta(z) \mu^{-1}  \frac{8 \rho }{ \left( \rho^{2}+1\right)^2}
\left( \varphi[\mu](r,t) +\psi_{*}(r,t) \right)
\\
&
+
\eta(4z) \bar{\mu}_0^{-1}
\mathcal{M}[\bar{\mu}_0]
\frac{\eta(\bar{\rho}) \mathcal{Z}(\bar{\rho})  }{\int_{0}^3 \eta(x) \mathcal{Z}^2(x) xdx }
+\eta(4z) \bar{\mu}_0^{-2}
\left(
-
\frac{\cos(2Q_\mu)-\cos(2Q_{\bar{\mu}_0})}{\bar{\rho}^2}
\Phi_e
\right)
\\
&
-
\eta(z) \frac{1}{2r^2}
\Big[ \sin[2( Q_\mu+\Phi_1+\Phi_2+\Psi_{*} + \eta(4z) \Phi_e )]
-
\sin(2Q_\mu)
-
\cos(2Q_\mu) 2\left( \Phi_1+\Phi_2+\Psi_{*} + \eta(4z) \Phi_e \right)
\Big]
\\
&
+
\frac{1}{2r^2}
\Big[
- \sin[2(\eta(z) Q_\mu+\Phi_1+\Phi_2+\Psi_{*} + \eta(4z) \Phi_e )]
+
\eta(z) \sin[2( Q_\mu+\Phi_1+\Phi_2+\Psi_{*} + \eta(4z) \Phi_e )]
\\
& + 2\left( 1-\eta(z)  \right) \left( \Phi_1+\Phi_2+\Psi_{*} \right)
\Big].
\end{align*}
\endgroup

Since
\begin{equation*}
\begin{aligned}
&
\eta(4z) \bar{\mu}_0^{-1}
\frac{-8 \bar\rho }{(\bar{\rho}^2+1)^2}
\left( \varphi[\bar{\mu}_0](r,t)
+
\psi_*(r,t)
\right)
+
\eta(z) \mu^{-1}  \frac{8 \rho }{ \left( \rho^{2}+1\right)^2}
\left( \varphi[\mu](r,t) +\psi_{*}(r,t) \right)
\\
= \ &
\left( \eta(4z)  -
\eta(z) \right) \bar{\mu}_0^{-1}
\frac{-8 \bar\rho }{(\bar{\rho}^2+1)^2}
\left( \varphi[\bar{\mu}_0](r,t)
+
\psi_*(r,t)
\right)
\\
& +
\eta(z)
\left(
 \mu^{-1}  \frac{8 \rho }{ \left( \rho^{2}+1\right)^2}
 -
 \bar{\mu}_0^{-1}
\frac{8 \bar\rho }{(\bar{\rho}^2+1)^2}
\right)
\left( \varphi[\bar{\mu}_0](r,t)
+
\psi_*(r,t)
\right)
\\
&
+
\eta(z) \mu^{-1}  \frac{8 \rho }{ \left( \rho^{2}+1\right)^2}
\left( \varphi[\mu](r,t) -
\varphi[\bar{\mu}_0](r,t)
 \right)
\end{aligned}
\end{equation*}
with $\mu=\bar{\mu}_0 + \mu_1$, we can write
\begin{equation}\label{def-E[v_1]}
\begin{aligned}
E[v_1]=&-\eta(4z)\pp_t \Phi_e+
\eta(z)
\left(
 \mu^{-1}  \frac{8 \rho }{ \left( \rho^{2}+1\right)^2}
 -
 \bar{\mu}_0^{-1}
\frac{8 \bar\rho }{(\bar{\rho}^2+1)^2}
\right)
\left( \varphi[\bar{\mu}_0](r,t)
+
\psi_*(r,t)
\right)
\\
&
+
\eta(z) \mu^{-1}  \frac{8 \rho }{ \left( \rho^{2}+1\right)^2}
\left( \varphi[\mu](r,t) -
\varphi[\bar{\mu}_0](r,t)
 \right) +E_\eta+E_e\\
 &
+
\eta(4z) \bar{\mu}_0^{-1}
\mathcal{M}[\bar{\mu}_0]
\frac{\eta(\bar{\rho}) \mathcal{Z}(\bar{\rho})  }{\int_{0}^3 \eta(x) \mathcal{Z}^2(x) xdx }
+\eta(4z) \bar{\mu}_0^{-2}
\left(
-
\frac{\cos(2Q_\mu)-\cos(2Q_{\bar{\mu}_0})}{\bar{\rho}^2}
\Phi_e
\right),
\end{aligned}
\end{equation}
where
\begin{equation}\label{def-E_eta}
\begin{aligned}
E_{\eta}:=&~ \Phi_e \eta'(4z)\frac{2r}{t^{\frac{3}{2}}}+\frac{16}{t}\eta''(4z)\Phi_e+\frac{8}{\sqrt t}\eta'(4z)\pp_r \Phi_e
+  \frac1r \frac4{\sqrt t}\eta'(4z)\Phi_e\\
&~+\left( \eta(4z)  -
\eta(z) \right) \bar{\mu}_0^{-1}
\frac{-8 \bar\rho }{(\bar{\rho}^2+1)^2}
\left( \varphi[\bar{\mu}_0](r,t)
+
\psi_*(r,t)
\right) .
\end{aligned}
\end{equation}

\medskip

\section{Gluing system}

\medskip

Having improved spatial decay by non-local corrections and time decay by solving the linearized elliptic equation, we are now ready to formulate the gluing system to deal with the remaining errors. We introduce the correction term
\begin{equation*}
	\Psi (r,t) + \eta_{R}(\rho)\phi (\rho,t)
	,
	\quad \rho =\mu^{-1} r
\end{equation*}
where $\eta$ is a smooth cut-off fucntion and $0\le \eta \le 1$, $\eta(s) =1$ for $s\le 1$ and $\eta(s) =0$ for $s\ge 2$; $\eta_{R}(\rho) = \eta(R^{-1} \rho)$ with $R$ depending on time and to be determined later. Recall that
\begin{equation*}
v_1(r,t) =\eta(z) Q_\mu+\Phi_1+\Phi_2+\Psi_{*} +\eta(4z)\Phi_e .
\end{equation*}
Then
\begingroup
\allowdisplaybreaks
\begin{align*}
	&
E\left[v_1 + \Psi (r,t) + \eta_{R}(\rho)\phi (\rho,t)  \right]
\\
= \ &
\pp_{rr} \left( \Psi + \eta_R(\rho) \phi(\rho,t) \right)
+
\frac{1}{r}
\pp_{r} \left( \Psi + \eta_R(\rho) \phi(\rho,t) \right)
-
\pp_{t} \left( \Psi + \eta_R(\rho) \phi(\rho,t) \right)
\\
&
+
\frac{\sin\left(2 v_1\right) -\sin\left( 2\left(v_1 + \Psi  + \eta_{R} \phi \right) \right)}{2r^2}
+
E\left[v_1 \right]
\\
= \ &
  \pp_{rr} \Psi +
\eta''(\frac{\rho}{R}) (\mu R)^{-2}
\phi
+ 2 (\mu R)^{-1} \mu^{-1}
\eta'(\frac{\rho}{R}) \pp_{\rho} \phi
+
\eta_R \mu^{-2} \pp_{\rho\rho}\phi
\\
&
+
\frac{1}{r}
\pp_{r} \Psi
+\eta'(\frac{\rho}{R}) \mu^{-2} (\rho R)^{-1} \phi+
\eta_R \mu^{-2} \frac{\pp_{\rho} \phi}{\rho}
\\
&
-
\pp_{t}\Psi
+
\eta'(\frac{\rho}{R}) \frac{\rho}{R}
\frac{(\mu R)'}{\mu R} \phi
-\eta_R \pp_{t} \phi
+
\eta_R \rho \pp_{\rho} \phi \mu^{-1} \dot{\mu}
\\
&
+
\frac{\sin\left(2v_1\right) -\sin\left( 2\left(v_1 + \Psi  + \eta_{R} \phi\right)  \right)}{2r^2}
+
E\left[v_1 \right]
\\
= \ &
-
\pp_{t}\Psi
+
\pp_{rr} \Psi
+
\frac{1}{r}
\pp_{r} \Psi
-
\frac{1}{r^2} \Psi
\\
&
+
\eta''(\frac{\rho}{R}) (\mu R)^{-2}
\phi
+\eta'(\frac{\rho}{R}) \mu^{-2} (\rho R)^{-1} \phi
+ 2 (\mu R)^{-1} \mu^{-1}
\eta'(\frac{\rho}{R}) \pp_{\rho} \phi
+
\eta'(\frac{\rho}{R}) \frac{\rho}{R}
\frac{(\mu R)'}{\mu R} \phi
+
\eta_R \rho \pp_{\rho} \phi \mu^{-1} \dot{\mu}
\\
&
-\eta_R \pp_{t} \phi
+
\eta_R \mu^{-2} \pp_{\rho\rho}\phi
+
\eta_R \mu^{-2} \frac{\pp_{\rho} \phi}{\rho}
-
\mu^{-2}
\frac{  \rho^{4} - 6 \rho^{2} + 1}{ \rho^2\left( \rho^{2} + 1  \right)^2}  \eta_{R} \phi
+
\eta_R \mu^{-2}  \frac{8}{ \left( \rho^{2}+1\right)^2} \Psi
\\
 &
+
(1-\eta_R) \mu^{-2} \frac{8}{ \left( \rho^{2}+1\right)^2}  \Psi
- \frac{\eta(z)-1}{r^2} \cos(2Q_{\mu}) \Psi
-
\mu^{-2} \left(\eta(z) -1 \right)
\frac{  \rho^{4} - 6 \rho^{2} + 1}{ \rho^2\left( \rho^{2} + 1  \right)^2}  \eta_{R} \phi
\\
& +
\frac{1}{2r^2}
\Big[ \sin\left( 2\left( \eta(z) Q_\mu+\Phi_1+\Phi_2+\Psi_{*} +\eta(4z)\Phi_e \right) \right)
-
\eta(z) \sin\left( 2\left(  Q_\mu+\Phi_1+\Phi_2+\Psi_{*} +\eta(4z)\Phi_e \right) \right)
\\
&+
\eta(z) \sin\left( 2\left(  Q_\mu+\Phi_1+\Phi_2+\Psi_{*} +\eta(4z)\Phi_e + \Psi  + \eta_{R} \phi \right) \right)
\\
&
-\sin\left( 2\left( \eta(z) Q_\mu+\Phi_1+\Phi_2+\Psi_{*} +\eta(4z)\Phi_e + \Psi  + \eta_{R} \phi \right) \right) \Big]
\\
& +
\frac{\eta(z)}{2r^2}
\Big\{ \sin\left( 2\left(  Q_\mu+\Phi_1+\Phi_2+\Psi_{*} +\eta(4z)\Phi_e \right) \right)
-\sin(2Q_{\mu})
-2\cos(2Q_{\mu}) \left(  \Phi_1+\Phi_2+\Psi_{*} +\eta(4z)\Phi_e \right)
\\
&
-
\Big[
\sin\left( 2\left(  Q_\mu+\Phi_1+\Phi_2+\Psi_{*} +\eta(4z)\Phi_e + \Psi  + \eta_{R} \phi \right) \right)
-\sin(2Q_{\mu})
\\
& -2\cos(2Q_{\mu}) \left( \Phi_1+\Phi_2+\Psi_{*} +\eta(4z)\Phi_e + \Psi  + \eta_{R} \phi \right)
\Big]
\Big\}  +
E\left[v_1 \right],
\end{align*}
\endgroup
where we have used
\begingroup
\allowdisplaybreaks
\begin{align*}
&
\frac{1}{2r^2}
\Big[ \sin\left( 2\left( \eta(z) Q_\mu+\Phi_1+\Phi_2+\Psi_{*} +\eta(4z)\Phi_e \right) \right)
\\
&
-\sin\left( 2\left( \eta(z) Q_\mu+\Phi_1+\Phi_2+\Psi_{*} +\eta(4z)\Phi_e + \Psi  + \eta_{R} \phi \right) \right) \Big]
\\
= \ &
- \frac{\eta(z)}{r^2} \cos(2Q_{\mu}) \left( \Psi  + \eta_{R} \phi \right)
\\
& +
\frac{1}{2r^2}
\Big[ \sin\left( 2\left( \eta(z) Q_\mu+\Phi_1+\Phi_2+\Psi_{*} +\eta(4z)\Phi_e \right) \right)
-
\eta(z) \sin\left( 2\left(  Q_\mu+\Phi_1+\Phi_2+\Psi_{*} +\eta(4z)\Phi_e \right) \right)
\\
&+
\eta(z) \sin\left( 2\left(  Q_\mu+\Phi_1+\Phi_2+\Psi_{*} +\eta(4z)\Phi_e + \Psi  + \eta_{R} \phi \right) \right)
\\
&
-\sin\left( 2\left( \eta(z) Q_\mu+\Phi_1+\Phi_2+\Psi_{*} +\eta(4z)\Phi_e + \Psi  + \eta_{R} \phi \right) \right) \Big]
\\
& +
\frac{\eta(z)}{2r^2}
\Big\{ \sin\left( 2\left(  Q_\mu+\Phi_1+\Phi_2+\Psi_{*} +\eta(4z)\Phi_e \right) \right)
-\sin(2Q_{\mu})
-2\cos(2Q_{\mu}) \left(  \Phi_1+\Phi_2+\Psi_{*} +\eta(4z)\Phi_e \right)
\\
&
-
\Big[
\sin\left( 2\left(  Q_\mu+\Phi_1+\Phi_2+\Psi_{*} +\eta(4z)\Phi_e + \Psi  + \eta_{R} \phi \right) \right)
-\sin(2Q_{\mu})
\\
& -2\cos(2Q_{\mu}) \left( \Phi_1+\Phi_2+\Psi_{*} +\eta(4z)\Phi_e + \Psi  + \eta_{R} \phi \right)
\Big]
 \Big\} .
\end{align*}
\endgroup
In order to make $E\left[v_1 + \Psi (r,t) + \eta_{R}(\rho)\phi (\rho,t)  \right]=0$, it suffices to solve the following gluing system.\\
\noindent $\bullet$ The outer problem:
\begin{equation}\label{out-eq}
\begin{aligned}
& \pp_{t}\Psi  =
\pp_{rr} \Psi
+
\frac{1}{r}
\pp_{r} \Psi
-
\frac{1}{r^2} \Psi+\mathcal G
\end{aligned}
\end{equation}
where
\begin{equation}\label{def-G}
\begin{aligned}
\mathcal G:=&~
(1-\eta_R) \mu^{-2} \frac{8}{ \left( \rho^{2}+1\right)^2}  \Psi
\\
&
+
\eta''(\frac{\rho}{R}) (\mu R)^{-2}
\phi
+\eta'(\frac{\rho}{R}) \mu^{-2} (\rho R)^{-1} \phi
+ 2 (\mu R)^{-1} \mu^{-1}
\eta'(\frac{\rho}{R}) \pp_{\rho} \phi
+
\eta'(\frac{\rho}{R}) \frac{\rho}{R}
\frac{(\mu R)'}{\mu R} \phi
\\
&
- \frac{\eta(z)-1}{r^2} \cos(2Q_{\mu}) \Psi
-
\mu^{-2} \left(\eta(z) -1 \right)
\frac{  \rho^{4} - 6 \rho^{2} + 1}{ \rho^2\left( \rho^{2} + 1  \right)^2}  \eta_{R} \phi
\\
& +
\frac{1}{2r^2}
\Big[ \sin\left( 2\left( \eta(z) Q_\mu+\Phi_1+\Phi_2+\Psi_{*} +\eta(4z)\Phi_e \right) \right)
-
\eta(z) \sin\left( 2\left(  Q_\mu+\Phi_1+\Phi_2+\Psi_{*} +\eta(4z)\Phi_e \right) \right)
\\
&+
\eta(z) \sin\left( 2\left(  Q_\mu+\Phi_1+\Phi_2+\Psi_{*} +\eta(4z)\Phi_e + \Psi  + \eta_{R} \phi \right) \right)
\\
&
-\sin\left( 2\left( \eta(z) Q_\mu+\Phi_1+\Phi_2+\Psi_{*} +\eta(4z)\Phi_e + \Psi  + \eta_{R} \phi \right) \right) \Big]
\\
& +
\frac{\eta(z)}{2r^2}
\Big\{ \sin\left( 2\left(  Q_\mu+\Phi_1+\Phi_2+\Psi_{*} +\eta(4z)\Phi_e \right) \right)
-\sin(2Q_{\mu})
-2\cos(2Q_{\mu}) \left(  \Phi_1+\Phi_2+\Psi_{*} +\eta(4z)\Phi_e \right)
\\
&
-
\Big[
\sin\left( 2\left(  Q_\mu+\Phi_1+\Phi_2+\Psi_{*} +\eta(4z)\Phi_e + \Psi  + \eta_{R} \phi \right) \right)
-\sin(2Q_{\mu})
\\
& -2\cos(2Q_{\mu}) \left( \Phi_1+\Phi_2+\Psi_{*} +\eta(4z)\Phi_e + \Psi  + \eta_{R} \phi \right)
\Big]
\Big\}  +
(1-\eta_R)E\left[v_1 \right],
\end{aligned}
\end{equation}
and $E\left[v_1 \right]$ is defined in \eqref{def-E[v_1]}.

\medskip

\noindent $\bullet$ The inner problem:
\begin{equation}
	\begin{aligned}
	 \mu^{2} \pp_{t} \phi  = &
	  \pp_{\rho\rho}\phi
		+
 \frac{\pp_{\rho} \phi}{\rho}
		-
		\frac{  \rho^{4} - 6 \rho^{2} + 1}{ \rho^2\left( \rho^{2} + 1  \right)^2}   \phi
		+
	  \dot{\mu}\mu  \rho \pp_{\rho} \phi
		+
	   \frac{8}{ \left( \rho^{2}+1\right)^2} \Psi
		+
		\mu^2  E\left[v_1 \right] ,
	\quad
		\rho \le 2R.
	\end{aligned}
\end{equation}
For the dealing of inner problem, it will be more convenient to use the $(\rho,\tau)$ variables with
\begin{equation}\label{def-tau}
\tau(t) =\int_{t_0}^{t} \mu^{-2}(s) ds + C_{\tau} t_0 \mu^{-2}(t_0)
\sim
\begin{cases}
t^{\gamma-1} (\ln t)^2 ,
&
1<\gamma<2
\\
t ,
&
\gamma=2
\\
t (\ln t)^2 ,
&
\gamma>2,
\end{cases}
\end{equation}
where $C_{\tau}$ is a large constant.

\medskip

\section{Orthogonal equation}\label{Sec-nonlocal}

\medskip

In this section, we formulate the orthogonal equation for $\mu_1$. Such orthogonality is required for finding well-behaved inner solution (see the linear theory given in  Appendix \ref{sec-inner}).
The orthogonal equation is given by
\begin{equation}\label{ortho-eq}
\int_0^{R_0} \left( \frac{8}{ \left( \rho^{2}+1\right)^2} \Psi
		+
		\mu^2  E\left[v_1 \right]  \right) \mathcal{Z}(\rho) \rho d\rho + O(R_0^{-\epsilon_0}) c_*\left[  \frac{8}{ \left( \rho^{2}+1\right)^2} \Psi
		+
		\mu^2  E\left[v_1 \right]  \right] =0
		\end{equation}
where $c_*$ is given in Proposition \ref{R0-linear}.

Notice that
\begin{equation*}
	\mu^{c}  \frac{\rho^a }{ \left( \rho^{2}+1\right)^b}
	-
	\bar{\mu}_0^{c}
	\frac{\bar\rho^a }{(\bar{\rho}^2+1)^b}
	=
	\mu^{c-1} \mu_1
	\left( 1+ O(\mu^{-1}\mu_1) \right)
	\frac{(2b-a+c)\rho^{a+2}+(c-a)\rho^a}{(\rho^2+1)^{b+1}}
\end{equation*}
since for
\begin{equation*}
	f(\theta)= \mu_{\theta}^{c}
	\frac{ \rho_{\theta}^a }{( \rho_{\theta}^2+1)^b} ,
	\quad
	\rho_{\theta} := \frac{r}{\mu_{\theta}}, \quad \mu_{\theta}:= \theta \mu + (1-\theta)\bar{\mu}_0
	=
	\mu -(1- \theta) \mu_1
\end{equation*}
we have
\begin{equation*}
	\begin{aligned}
		f'(\theta)= \ & c\mu_{\theta}^{c-1} \mu_1
		\frac{ \rho_{\theta}^a }{( \rho_{\theta}^2+1)^b}
		+
		\mu_{\theta}^{c-1} \mu_1
		\frac{\rho_\theta^a \left[ (2b-a) \rho_\theta^2-a\right] }{\left( \rho_\theta^2 +1\right)^{b+1} }
		\\
		= \ &
		\mu^{c-1} \mu_1
		\left( 1+ O(\mu^{-1}\mu_1) \right)
		\frac{(2b-a+c)\rho^{a+2}+(c-a)\rho^a}{(\rho^2+1)^{b+1}}  .
	\end{aligned}
\end{equation*}
By \eqref{varphi+psi} and Proposition \ref{varphi-coro}, for $R_0 \mu \ll t^{\frac{1}{2}}$, we have
\begingroup
\allowdisplaybreaks
\begin{align*}
&
\int_0^{R_0}
\left[    \mu^{-1}  \frac{8 \rho }{ \left( \rho^{2}+1\right)^2}
\left( \varphi[\mu](r,t) +\psi_{*}(r,t) \right)
-
 \bar{\mu}_0^{-1}
\frac{8 \bar\rho }{(\bar{\rho}^2+1)^2}
\left( \varphi[\bar{\mu}_0](r,t)
+
\psi_*(r,t)
\right)
\right]
\mathcal{Z} (\rho) \rho d\rho
\\
= \ &
\int_0^{R_0}
\Big[     \frac{8 \mu^{-1}  \rho }{ \left( \rho^{2}+1\right)^2}
\left( \varphi[\mu](r,t) -\varphi[\bar{\mu}_0](r,t)  \right)
+
\left(
 \frac{8 \mu^{-1}  \rho }{ \left( \rho^{2}+1\right)^2}
-
\frac{8 \bar{\mu}_0^{-1}  \bar\rho }{(\bar{\rho}^2+1)^2}
\right)
\left( \varphi[\bar{\mu}_0](r,t)
+
\psi_*(r,t)
\right)
\Big]
\mathcal{Z} (\rho) \rho d\rho
\\
= \ &
\int_0^{R_0}
\Bigg\{     \frac{8 \mu^{-1}  \rho }{ \left( \rho^{2}+1\right)^2}
\bigg[
-2^{-1} \Big( \mu_1  t^{-1}
+
\int_{t/2}^{t-\mu_{0}^2 } \frac{\dot\mu_{1} (s)}{t-s} d s
\Big)
\\
&
+ O\Big( |\mu_1| t^{-2} r^2 +  |\dot{\bar{\mu}}_0| \sup\limits_{t_{1}\in[t/2,t]}\Big(\frac{|\mu_1(t_1)|}{ \bar{\mu}_0(t)} + \frac{|\dot\mu_{1}(t_1)|}{|\dot{\bar{\mu}}_0(t)|}  \Big)  \frac{ r}{\bar{\mu}_0} \Big)
+
\tilde{g}[\bar{\mu}_0,\mu_1]
\bigg]
\\
&
+
8 \mu^{-2} \mu_1
\left( 1+ O(\mu^{-1}\mu_1) \right)
\frac{2\rho^3-2\rho}{\left( \rho^2 +1\right)^{3} }
\Big[
 O( \mu_0  t^{-2} r^2 )
+
|\dot{\mu}_0|
\min\{ \langle \rho \rangle,\ln t \}
\\
&
+ O( \mu_0 \rho  t^{-\frac{1}{2}} v_{\gamma}(t) )
+
O\left( \ln \ln t |\dot{\mu}_0| \right)
\Big]
\Bigg\}
\frac{\rho^2}{\rho^2 +1}  d\rho
\\
= \ &
\mu^{-1}
\int_0^{R_0}
\Bigg\{     \frac{8   \rho }{ \left( \rho^{2}+1\right)^2}
\bigg[
-2^{-1} \Big( \mu_1  t^{-1}
+
\int_{t/2}^{t-\mu_{0}^2 } \frac{\dot\mu_{1} (s)}{t-s} d s
\Big)
\\
&
+ O\Big( |\mu_1| t^{-2} \mu_0^2 \rho^2 +  |\dot{\bar{\mu}}_0| \sup\limits_{t_{1}\in[t/2,t]}\Big(\frac{|\mu_1(t_1)|}{ \bar{\mu}_0(t)} + \frac{|\dot\mu_{1}(t_1)|}{|\dot{\bar{\mu}}_0(t)|}  \Big)  \rho \Big)
+
\tilde{g}[\bar{\mu}_0,\mu_1]
\bigg]
\\
&
+
8 \mu^{-1} \mu_1
\left( 1+ O(\mu^{-1}\mu_1) \right)
\frac{2\rho^3-2\rho}{\left( \rho^2 +1\right)^{3} }
\Big[
O( \mu_0  t^{-2}\mu_0^2 \rho^2 )
+
|\dot{\mu}_0|
\min\{ \langle \rho \rangle,\ln t \}
\\
&
+ O( \mu_0 \rho  t^{-\frac{1}{2}} v_{\gamma}(t) )
+
O\left( \ln \ln t |\dot{\mu}_0| \right)
\Big]
\Bigg\}
\frac{\rho^2}{\rho^2 +1}  d\rho
\\
= \ &
\mu^{-1}
\Bigg\{
-\left( 1+O(R_0^{-2})\right) \Big( \mu_1  t^{-1}
+
\int_{t/2}^{t-\mu_{0}^2 } \frac{\dot\mu_{1} (s)}{t-s} d s
\Big)
\\
&
+ O\Big( |\mu_1| t^{-2} \mu_0^2 \ln R_0
+
  |\dot{\mu}_0| \sup\limits_{t_{1}\in[t/2,t]}\Big(\frac{|\mu_1(t_1)|}{ \mu_0(t)} + \frac{|\dot\mu_{1}(t_1)|}{|\dot{\mu}_0(t)|}  \Big)   \Big)
\\
& +
O\Big(|\dot\mu_0  | \ln t
\sup\limits_{t_{1}\in[t/2,t]} \Big(\frac{|\mu_1(t_1)|}{\mu_0(t)}
+ \frac{|\dot\mu_{1}(t_1)|}{|\dot\mu_0(t)|}\Big)^2\Big)
\\
&
+ 	O \bigg(  t^{-2} \int_{t_0/2}^{t}
\Big( s^{-1} |\mu_1(s)| \mu_0^{2}(s) + s |\dot\mu_0(s)| \Big(\frac{|\mu_1(s)|}{ \mu_0(s) } + \frac{|\dot\mu_{1}(s)|}{|\dot\mu_0(s)|}  \Big) \Big)  ds \bigg)
\\
&
+
\mu^{-1} \mu_1
\left( 1+ O(\mu^{-1}\mu_1) \right)
\Big[
O\left( \mu_0  t^{-2}\mu_0^2 \ln R_0 + |\dot{\mu}_0| \right)
+ O\left( \mu_0  t^{-\frac{1}{2}} v_{\gamma}(t) \right)
+
O\left( \ln \ln t |\dot{\mu}_0| \right)
\Big]
\Bigg\}  ,
\end{align*}
\endgroup
\begin{equation*}
	\big|   \bar{\mu}_0^{-1}
	\mathcal{M}[\bar{\mu}_0]  \big|
	\lesssim \mu_0^{-1} t^{-2}.
\end{equation*}
By \eqref{Phie-est1}, one has
\begin{equation*}
		\int_0^{R_0} \eta(4z) \bar{\mu}_0^{-2}
		\left(
		-
		\frac{\cos(2Q_\mu)-\cos(2Q_{\bar{\mu}_0})}{\bar{\rho}^2}
		\Phi_e
		\right) \mathcal{Z} (\rho) \rho d\rho
		=
		\int_0^{R_0}
		O \left( \mu_1 \mu_0^{-3} |\Phi_e|
		\rho^2 \langle \rho \rangle^{-6}  \right) d\rho
		=
		O \left( \mu_1 \mu_0^{-3}  |\dot{\mu}_0| \ln\ln t \right).
\end{equation*}

The contribution from the outer problem is given  by
\begin{equation*}
\int_0^{R_0}\frac{8}{ \left( \rho^{2}+1\right)^2} \lambda \rho \psi(\lambda \rho,t)  \mathcal{Z} (\rho) \rho d\rho
=
\lambda   \int_0^{R_0}\frac{8\rho^3}{ \left( \rho^{2}+1\right)^3} \psi(\lambda \rho,t)  d\rho .
\end{equation*}

Using the proposition in the next subsection, we will be able to solve \eqref{ortho-eq} with the estimate $|\dot\mu_{1}| \lesssim \vartheta(t)  $.

\medskip

\subsection{A linear problem for the orthogonal equation}\label{Sec-linearnonlocal}

\medskip

We consider a model problem for the orthogonal equation:
\begin{equation}\label{linear-nonlocal}
	\int_{t/2}^{t-\mu_0^2(t)}\frac{\dot\mu_1(s)}{t-s}ds +\frac{\mu_1(t)}{t}=a_1[\mu](t)+
	a_2\left[\mu_1,\dot{\mu}_1 \right](t)
	+
	a_3\left[\mu_1,\dot{\mu}_1 \right](t)
\end{equation}
with $\mu=\bar\mu_0+\mu_1$, where $a_1$ represents the new error after adding elliptic correction $\Phi_e$. For $p\ne -1$, we define the following norm for $a_1[\mu](t)$
\begin{equation}\label{topo-a1}
	\|a_1\|_{p}:=\sup_{ t\ge t_0, \mu_1\in B_{\mu_1}} \left[ \left(t^{p} \ln t \right)^{-1} |a_1[\mu](t)|+ t|\pp_{\mu} a_1[\mu](t)|\right] ,
\end{equation}

\begin{equation}%\label{topo-a1}
	[a_1]_{p,\alpha}:=\sup_{ t_1,t_2\in [t/2,t],  \mu_1\in B_{\mu_1} }   \left(t^{p-\alpha} \ln t \right)^{-1}  \frac{|a_1[\mu](t_1)-
		a_1[\mu](t_2) |}{|t_1-t_2|^{\alpha}}    ,
\end{equation}
for some  $0<\alpha<1$.
\begin{equation}
	\begin{aligned}
		&	a_2[\mu_1 ,\dot{\mu}_1 ](t) =
		O\Big( |\mu_1| t^{-2} \mu_0^2 \ln R_0
		+
		|\dot{\mu}_0| \sup\limits_{t_{1}\in[t/2,t]}\Big(\frac{|\mu_1(t_1)|}{ \mu_0(t)} + \frac{|\dot\mu_{1}(t_1)|}{|\dot{\mu}_0(t)|}  \Big)   \Big)
		\\
		& +
		O\Big(|\dot\mu_0  | \ln t
		\sup\limits_{t_{1}\in[t/2,t]} \Big(\frac{|\mu_1(t_1)|}{\mu_0(t)}
		+ \frac{|\dot\mu_{1}(t_1)|}{|\dot\mu_0(t)|}\Big)^2\Big)
		\\
		&
		+ 	O \bigg(  t^{-2} \int_{t_0/2}^{t}
		\Big( s^{-1} |\mu_1(s)| \mu_0^{2}(s) + s |\dot\mu_0(s)| \Big(\frac{|\mu_1(s)|}{ \mu_0(s) } + \frac{|\dot\mu_{1}(s)|}{|\dot\mu_0(s)|}  \Big) \Big)  ds \bigg)
		\\
		&
		+
		\mu^{-1} \mu_1
		\left( 1+ O(\mu^{-1}\mu_1) \right)
		\Big[
		O\left( \mu_0  t^{-2}\mu_0^2 \ln R_0 + |\dot{\mu}_0| \right)
		+ O\left( \mu_0  t^{-\frac{1}{2}} v_{\gamma}(t) \right)
		+
		O\left( \ln \ln t |\dot{\mu}_0| \right)
		\Big] ,
	\end{aligned}
\end{equation}
\begin{equation*}
	\left|
	a_3\left[\mu_1,\dot{\mu}_1 \right](t)  \right|
	\lesssim R_0^{-\epsilon_0}
	\left( \left|a_1[\mu_1](t) \right| +
	\left|a_2\left[\mu_1,\dot{\mu}_1 \right](t) \right|
	+
	\int_{t/2}^{t-\mu_0^2(t)}\frac{|\dot\mu_1(s)|}{t-s}ds +\frac{|\mu_1(t)|}{t}
	\right) .
\end{equation*}

For some $0<\nu<1$, we write
\begin{equation}
	\begin{aligned}
		\int_{t/2}^{t-\mu_0^2(t)}\frac{\dot\mu_1(s)}{t-s}ds=&~\int_{t/2}^{t-t^{1-\nu}}\frac{\dot\mu_1(s)}{t-s}ds+\int_{ t-t^{1-\nu}  }^{t-\mu_{0}^2(t) } \frac{\dot\mu_1(s) - \dot\mu_1(t)}{t-s} d s\\
		&~+\dot\mu_1(t)\left[(1-\nu)\ln t-2\ln \mu_0(t)\right] .
	\end{aligned}
\end{equation}
We leave the partial error
\begin{align*}%\label{E-mu1}
	\mathcal{E}_{\nu}[\mu_{1}] := \int_{ t-t^{1-\nu}  }^{t-\mu_{0}^2(t) } \frac{\dot\mu_{1}(s) - \dot\mu_{1}(t)}{t-s} d s
\end{align*}
to the nonorthogonal inner problem and
consider
\begin{equation}\label{mu1-eq}
	\begin{aligned}
		\dot\mu_1(t)
		=  \ &
		\left[(1-\nu)\ln t-2\ln \mu_0(t)\right]^{-1}
		\Bigg(
		-
		t^{-1} \mu_1
		-
		\int_{t/2}^{t-t^{1-\nu}}\frac{\dot\mu_1(s)}{t-s}ds
		\\
		&
		+
		a_1[\mu](t)+
		a_2\left[\mu_1,\dot{\mu}_1 \right](t)
		+
		a_3\left[\mu_1,\dot{\mu}_1 \right](t)
		\Bigg) ,\quad t\ge t_0.
	\end{aligned}
\end{equation}

\begin{prop}\label{prop-toynonlocal}
For $\mu_0$ given in \eqref{def-mu_0}. Suppose $p\ne -1$,
	$p<-\frac{\gamma}{2}$ if $1<\gamma<2$, and $p<-1$ if $\gamma\ge 2$; $2\nu <\min\left\{ \gamma-1,1\right\}$; $\|a_1\|_{p} \le C_{a_1}$ for a constant $C_{a_1}\ge 1$.
	Then for $R_0, t_0$ sufficiently large, there exists a unique solution $\mu_1$ to \eqref{mu1-eq} satisfying
	\begin{equation*}
|\dot\mu_{1}(t)| \lesssim t^p \|a_1\|_{p},\quad
\mu_1=
\begin{cases}
	\int_{t_0}^{t} \dot{\mu}_1(s) ds,
	&
	p > -1
	\\
	-\int_{t}^\infty
	\dot{\mu}_1(s) ds,
	& p <-1 .
\end{cases}
	\end{equation*}
Moreover, if $[a_1]_{p,\alpha}<\infty$, then
\begin{equation}%\label{topo-a1}
\frac{| \dot\mu_{1}(t_1)-
	\dot\mu_{1}(t_2) |}{|t_1-t_2|^{\alpha}}
\lesssim
t^{p-\alpha} \left(\|a_1\|_{p} + [a_1]_{p,\alpha} \right)
  ~\mbox{ \ for \ }~ t_1,t_2\in[\frac{t}{2},t] .
\end{equation}

\end{prop}

\begin{proof}
Notice that
\begin{equation*}
\left[(1-\nu)\ln t-2\ln \mu_0(t)\right]^{-1}
= C_{\gamma,\nu} (\ln t)^{-1}
\left( 1+ O(\frac{\ln\ln t}{\ln t})\right) ,
\quad
C_{\gamma,\nu}:= \left( \min\left\{\gamma-1,1\right\}-\nu \right)^{-1} ,
\end{equation*}
\begin{equation*}
		\left[(1-\nu)\ln t-2\ln \mu_0(t)\right]^{-1} |a_1 [\mu](t)|
		\le
		C_{\gamma,\nu}
		\left( 1+ O(\frac{\ln\ln t}{\ln t})\right)
		t^{p}  \|a_1\|_{p} .
\end{equation*}
From the above estimate, we introduce the norm
\begin{equation}\label{norm-mu1}
	\| \mu_{1} \|_{*}:= \sup\limits_{t\ge t_0/4}
\left( C_{\gamma,\nu}  t^{p} \right)^{-1} |\dot\mu_{1}(t)|
\end{equation}
and $\dot{\mu}_1$ will be solved in the space
\begin{equation}\label{mu1-topo}
	B_{\mu_1} := \{ \mu_1 \in C^{1}(t_0/4,\infty) \ : \ \| \mu_{1} \|_{*} \le C_{\mu_1} \|a_1\|_{p}  \}
\end{equation}
for a large constant $C_{\mu_1} \ge 2$ to be determined later.
Denote
\begin{equation*}
	I[\dot{\mu}_1]:=
	\begin{cases}
		\int_{t_0}^{t} \dot{\mu}_1(s) ds,
		&
		p\ge -1
		\\
		-\int_{t}^\infty
		\dot{\mu}_1(s) ds,
		& p <-1 .
	\end{cases}
\end{equation*}
In order to solve \eqref{mu1-eq}, it suffices to consider the following fixed point problem about $\dot{\mu}_1$.
\begin{equation}
\begin{aligned}
\mathcal{S}[\dot{\mu}_1](t)
=  \ & \chi(t)
\left[(1-\nu)\ln t-2\ln \mu_0(t)\right]^{-1}
\Big(
-
t^{-1} I[\dot{\mu}_1]
-
\int_{t/2}^{t-t^{1-\nu}}\frac{\dot\mu_1(s)}{t-s}ds
\\
&
+
a_1[\bar{\mu}_0+I[\dot{\mu}_1] ](t)+
a_2\left[I[\dot{\mu}_1],\dot{\mu}_1 \right](t)
+
a_3\left[I[\dot{\mu}_1],\dot{\mu}_1 \right](t)
\Big) ~\mbox{ \ for \ }~ t\ge \frac{t_0}{4},
\end{aligned}
\end{equation}
where $\chi(t)$ is a smooth cut-off function such that $\chi(t)= 0$ for $t< \frac{3}{4} t_0$ and $\chi(t) = 1$ for $t\ge t_0$.
For any $\dot{\mu}_1 \in B_{\mu_1}$, since $p\ne -1$, we have
\begin{equation*}
	|  I[\dot{\mu}_1] |
	\le
	C_{\gamma,\nu} \| \mu_{1} \|_{*}
	|p+1|^{-1} t^{p+1},
	\quad
|t^{-1} I[\dot{\mu}_1] |
\le
C_{\gamma,\nu} \| \mu_{1} \|_{*}
|p+1|^{-1} t^{p} .
\end{equation*}
\begin{equation*}
	\begin{aligned}
		&
		\Big| \int_{t/2}^{t-t^{1-\nu}}\frac{\dot\mu_1(s)}{t-s}ds  \Big|
		\le
		C_{\gamma,\nu} \| \mu_{1} \|_{*}
		\int_{t/2}^{t-t^{1-\nu}}\frac{s^p}{t-s}ds
		=
		C_{\gamma,\nu} \| \mu_{1} \|_{*}
		t^p
		\int_{1/2}^{1-t^{-\nu}} \frac{x^p}{1-x} dx
		\\
		= \ &
		C_{\gamma,\nu} \| \mu_{1} \|_{*}
		t^p
		\left(
		\int_{1/2}^{1-t^{-\nu}} \frac{x^p -1}{1-x} dx
		+\nu \ln t -\ln 2
		\right)
		=
		C_{\gamma,\nu} \| \mu_{1} \|_{*}
		t^p
		\nu \ln t
		\left( 1+ O( (\ln t)^{-1} )\right) .
	\end{aligned}
\end{equation*}
For the terms in $a_2[\mu_1 ,\dot{\mu}_1 ]$, we have
\begin{equation*}
\left| | I[\dot{\mu}_1] | t^{-2} \mu_0^2 \ln R_0  \right|
\lesssim
C_{\gamma,\nu} \| \mu_{1} \|_{*} t^{-\epsilon_1} t^{p}
\end{equation*}
for an $\epsilon_1>0$ sufficiently small due to $\gamma>1$.
\begin{equation*}
\begin{aligned}
|\dot{\mu}_0| \sup\limits_{t_{1}\in[t/2,t]}\Big(\frac{|I[\dot{\mu}_1](t_1)|}{ \mu_0(t)} + \frac{|\dot\mu_{1}(t_1)|}{|\dot{\mu}_0(t)|}  \Big)
\lesssim
C_{\gamma,\nu} \| \mu_{1} \|_{*} t^{p}  .
\end{aligned}
\end{equation*}
\begin{equation*}
|\dot\mu_0  | \ln t
\sup\limits_{t_{1}\in[t/2,t]} \Big(\frac{|  I[\dot{\mu}_1] (t_1)|}{\mu_0(t)}
+ \frac{|\dot\mu_{1}(t_1)|}{|\dot\mu_0(t)|}\Big)^2
\lesssim
C_{\gamma,\nu} \| \mu_{1} \|_{*} t^{-\epsilon_1} t^{p}
\end{equation*}
for an $\epsilon_1>0$ small due to the choice $p<-\frac{\gamma}{2}$ when $1<\gamma<2$ and $p<-1$ when $\gamma\ge 2$.
\begin{equation*}
\begin{aligned}
&
t^{-2} \int_{t_0/2}^{t}
\Big( s^{-1} |I[\dot{\mu}_1](s)| \mu_0^{2}(s) + s |\dot\mu_0(s)| \Big(\frac{|I[\dot{\mu}_1](s)|}{ \mu_0(s) } + \frac{|\dot\mu_{1}(s)|}{|\dot\mu_0(s)|}  \Big) \Big)  ds
\\
\lesssim \ &
t^{-2} \int_{t_0/2}^{t}
  \left( |I[\dot{\mu}_1](s)| + s |\dot\mu_{1}(s)| \right)    ds
\lesssim
	C_{\gamma,\nu} \| \mu_{1} \|_{*} t^{p}
\end{aligned}
\end{equation*}
since $p>-2$.
\begin{equation*}
\begin{aligned}
&
\mu^{-1} |  I[\dot{\mu}_1] |
\left( 1+ O(\mu^{-1}\mu_1) \right)
\Big[
O\left( \mu_0  t^{-2}\mu_0^2 \ln R_0 + |\dot{\mu}_0| \right)
+ O\left( \mu_0  t^{-\frac{1}{2}} v_{\gamma}(t) \right)
+
O\left( \ln \ln t |\dot{\mu}_0| \right)
\Big]
\\
\lesssim \ &
C_{\gamma,\nu} \| \mu_{1} \|_{*}
t^{p} \ln\ln t   .
\end{aligned}
\end{equation*}

For $	a_3\left[\mu_1,\dot{\mu}_1 \right]$, we have
\begin{equation*}
\int_{t/2}^{t-\mu_0^2(t)}\frac{|\dot\mu_1(s)|}{t-s}ds
\le
C_{\gamma,\nu} \| \mu_{1} \|_{*}
\int_{t/2}^{t-\mu_0^2(t)}\frac{s^p}{t-s}ds
\lesssim
C_{\gamma,\nu} \| \mu_{1} \|_{*}  t^p\ln t .
\end{equation*}

\begin{equation}
	\begin{aligned}
		|\mathcal{S}[\dot{\mu}_1](t) |
		\le  \ & \chi(t)
		C_{\gamma,\nu} (\ln t)^{-1}
		\left( 1+ O(\frac{\ln\ln t}{\ln t})\right)
		\Big\{
		\left(1+O(R_0^{-\epsilon_1})\right)
		\Big[
		C_{\gamma,\nu} \| \mu_{1} \|_{*}
		|p+1|^{-1} t^{p}
		\\
		& +
		C_{\gamma,\nu} \| \mu_{1} \|_{*}
		t^p
		\nu \ln t
		\left( 1+ O( (\ln t)^{-1} )\right)
		\\
		&
		+
		t^p \ln t \|a_1\|_{p } +
		C_2 C_{\gamma,\nu} \| \mu_{1} \|_{*}
		t^{p} \ln\ln t
		\Big]
		+
		C_3 R_0^{-\epsilon_1} C_{\gamma,\nu} \| \mu_{1} \|_{*}  t^p\ln t
		\Big\} ~\mbox{ \ for \ }~ t\ge \frac{t_0}{4} .
	\end{aligned}
\end{equation}
With above estimates, we then proceed as follows. First, we take
$\nu C_{\gamma,\nu}  <1$ and choose $C_{\mu_1} $ sufficiently large such that $\nu C_{\gamma,\nu} C_{\mu_1} + 1<C_{\mu_1}$. Then we take $R_0$ large enough. Finally, we take $t_0$ sufficiently large. Choosing the above parameters, we have $\mathcal{S}[\dot{\mu}_1](t)  \in B_{\mu_1}$.

For the contraction property, most terms can be verified similarly. Let us focus on the continuity of $a_1[\mu_1](t)$ about $\mu_1$. For any $\mu_{1a}, \mu_{1b}\in B_{\mu_1}$, we have
\begin{equation*}
	\left|
a_1[\bar{\mu}_0+I[\dot{\mu}_{1a}]](t)
-
a_1[\bar{\mu}_0+I[\dot{\mu}_{1b}]](t)
\right|
\lesssim t^{-1} \left| I[\dot{\mu}_{1a}] - I[\dot{\mu}_{1b}] \right|
\|a_1\|_{p}
\lesssim
C_{\gamma,\nu} \|\mu_{1a}-\mu_{1b}\|_{*}
|p+1|^{-1} t^{p}
\|a_1\|_{p}
\end{equation*}
and $\left[(1-\nu)\ln t-2\ln \mu_0(t)\right]^{-1}$ provides small quantity when $t_0$ is large.

\end{proof}

\medskip

\section{Weighted topologies and solving the full problem}

\medskip

Let us first fix the inner solution $\phi$ to the inner problem, and the next order of scaling parameter $\mu_1$ in the spaces with the following norms
\begin{equation}\label{norm-phi}
	\| \phi \|_{i,\kappa,a} :=
\sup\limits_{(\rho,\tau)\in \DD_{4R}}
\tau^{\kappa}
\langle \rho\rangle^{a}
\left( \langle \rho\rangle |\pp_\rho\phi(\rho,t(\tau) )| + |\phi(\rho,t(\tau) )| \right)
\end{equation}
for some positive constants $a\in(0,1)$, $\kappa$  to be determined later.

 For $ \mu_1(t)\in C^{1}(\frac{t_0}{4},\infty)$, $\mu_{1}(t) \rightarrow 0$ as $t\rightarrow\infty$, denote
\begin{equation}\label{norm-mu}
\| \mu_{1} \|_{*1}:= \sup\limits_{t\ge t_0/4}
\left[  \vartheta(t) \right]^{-1} |\dot\mu_{1}|
\end{equation}
with the weighted function
\begin{equation}
\vartheta(t):= \tau^{-\kappa}(t) \mu_0^{-1}(t)R^{-1-a}(t).
\end{equation}
Here,  in order to restrict the inner problem in the self-similar region, it is reasonable to assume that
$$
\mu_0 R\ll \sqrt t.
$$
Let us write
\begin{equation}
R(t)= t^{\omega},
\mbox{ \ where  \ }
0<\omega<\begin{cases}
\frac{\gamma-1}{2},\quad& 1<\gamma<2\\
\frac12,\quad& \gamma\geq 2 .
\end{cases}
\end{equation}

Set $\Psi = r\psi(r,t)$. In order to find a solution $\Psi$ for the outer problem \eqref{out-eq}, it is equivalent to find a fixed point about $\psi$ for the following problem:
\begin{equation*}
	\psi(x,t)=T_4 \bullet \left[ r^{-1} \mathcal{G}[r\psi,\phi,\mu]\right] (x,t,t_0) .
\end{equation*}

We define
\begin{equation}\label{norm-out}
	\|\psi\|_{{\rm out}}= \sup_{r\in (0,\infty
	),t \ge t_0} \left[\left( w_{o}(r,t) \right)^{-1}|\psi(r,t)| \right],
\end{equation}
where
\begin{equation*}
w_{o}(r,t):= \vartheta(t)
\left(\1_{\{ r\le t^{\frac{1}{2}} \}} + t r^{-2}  \1_{\{ r> t^{\frac{1}{2}} \}} \right)
\end{equation*}
Above fixed point problem will be solved in Appendix \ref{Sec-outer}. Once we have pointwise estimate, gradient and H\"older estimates can be obtained by scaling argument. In fact, we show in Appendix \ref{Sec-outer} that
\begin{align*}
|\psi|\lesssim w_{o}(r,t),
\quad
\sup_{t_1,t_2\in (t-\frac{\la^2(t)}{4},t)}\frac{|\psi(r,t_1)-\psi(r,t_2)|}{|t_1-t_2|^{\alpha}}\lesssim
\vartheta(t)
 \left[\la^{-2\alpha}+\la^{2-2\alpha}(\mu_0 R)^{-2}
 \right]
\end{align*}
where $0<\alpha<1$ and $0<\la(t)\leq \sqrt t$. Later we will choose $\la(t)=\sqrt t$.

\medskip

Now by using the H\"older property of $\psi$, we control the remaining term $
\mu\mathcal{E}_{\nu}[\mu_{1}]
$
that we put in the non-orthogonal part of the inner problem. Similar to the process in \cite[Section 4.2]{infi4D}, we have
\begin{equation}
 [\dot\mu_{1}]_{  C^{\alpha}(\frac{3t}{4},t)  } \lesssim [t^{-\alpha}+t^{1-\alpha}(\mu R)^{-2}] \vartheta(t)
\end{equation}
by taking $\la(t)=\sqrt t$ in \eqref{norm-out}. Then we have
\begin{equation}
\begin{aligned}
|\mu\mathcal{E}_{\nu}[\mu_{1}]|\lesssim&~ \mu_0 [\dot\mu_{1}]_{  C^{\alpha}(\frac{3t}{4},t)  } \max\{\mu_0^{2\alpha},t^{(1-\nu)\alpha}\}\\
\lesssim&~\mu_0 [t^{-\alpha}+t^{1-\alpha}(\mu_0 R)^{-2}] \mu^{a}\tau^{-\kappa}(\mu_0 R)^{-1-a}\max\{\mu_0^{2\alpha},t^{(1-\nu)\alpha}\}
\end{aligned}
\end{equation}
By Proposition \ref{linear-theory} and Proposition \ref{R0-linear}, we need
$$
R^2\ln R\mu_0 [t^{-\alpha}+t^{1-\alpha}(\mu_0 R)^{-2}] \mu_0^{a}\tau^{-\kappa}(\mu_0 R)^{-1-a}\max\{\mu_0^{2\alpha},t^{(1-\nu)\alpha}\}
\ll v(t)\ll \tau^{-\kappa} R_0^{\ell-6}(\ln R_0)^{-1},
$$
i.e.,
$$
R^{1-a}\ln R t^{1-\alpha}(\mu_0 R)^{-2}\max\{\mu_0^{2\alpha},t^{(1-\nu)\alpha}\}
\ll 1,
$$
where we have used
$$
\mu_0 R\ll \sqrt t.
$$
We then require
\begin{equation}\label{eqn-912912912}
\begin{cases}
R^{-1-a}\ln R \mu_0^{2\alpha-2} t^{1-\alpha}\ll 1,\\
R^{-1-a}\ln R \mu_0^{-2} t^{1-\alpha+(1-\nu)\alpha}\ll 1.\\
\end{cases}
\end{equation}
 Recall the definition of $\mu_0$ in \eqref{def-mu_0}.
 Then
we have
\begin{equation}\label{restriction1}
\begin{cases}
\begin{cases}
-\omega(1+a)+\gamma-1-\nu\alpha<0\\
0<\omega<\frac{\gamma-1}{2}\\
0<\nu<\frac{\gamma-1}{2}\\
\end{cases}
,~&~\mbox{ if }~ 1<\gamma<2\\
\begin{cases}
-\omega(1+a)+1-\nu\alpha<0\\
0<\omega<\frac12\\
0<\nu<\frac12\\
\end{cases}
,~&~\mbox{ if }~ \gamma\geq 2,\\
\begin{cases}
0<\alpha<1\\
0<a<\ell-2\\
\end{cases}, &\quad \gamma>1,
\end{cases}
\end{equation}
where the last restriction is from the need in the inner problem, see Proposition \ref{R0-linear}.

\medskip

\begin{proof}[Proof of Theorem \ref{thm}]
After the weighted spaces are fixed for $(\psi,\phi,\mu_1)$, the fixed point argument can be then carried out by using the linear theories, where the linear theory for the inner problem is proved in Appendix \ref{sec-inner}, the solvability of $\mu_1$ is showed in Section \ref{Sec-linearnonlocal} by controlling a non-local remainder in the inner problem, and the linear theory for the outer problem corresponds essentially to convolutions in $\R^4$ (cf. \cite[Appendix A]{infi4D}). The  inner and outer problems are analyzed in Appendix \ref{Sec-RHS}, and the contraction mapping theorem can be applied if one can choose constants satisfying the constraints \eqref{restriction1}, \eqref{out-restrict} and  \eqref{restriction2}. For $\gamma\geq 2$, this is straightforward, and for $1<\gamma<2$, one has valid choices in the entire range with the aid of Mathematica. The proof is thus complete.
\end{proof}

\bigskip

\medskip

\appendix

\section{Analyzing the gluing system}\label{Sec-RHS}

\medskip

To estimate the errors appearing in the RHS of inner and outer problems, we recall that we measure
\begin{itemize}
\item
$\mu_1$ with the norm \eqref{norm-mu}.
\item the RHS for the inner problem with the $\|\cdot\|_{v,\ell}$-norm where
$$
v(t)=\tau^{-\kappa} R_0^{-5}, \quad 1<\ell<3,
$$
and $R_0>0$ is a large constant.
\item the inner solution $\phi$ with the $\|\cdot\|_{i,\kappa,a}$-norm defined in \eqref{norm-phi}.
\item the outer solution $\psi$ with the $\|\cdot\|_{{\rm out}}$-norm defined in \eqref{norm-out}.
\end{itemize}

\medskip

We first give some estimates for the terms in $E[v_1]$ defined in \eqref{def-E[v_1]}. Notice that the support of $E_\eta$ (defined in \eqref{def-E_eta}) in outside the inner region.

\noindent $\bullet$
\begin{equation*}
\begin{aligned}
\left| \eta(4z)\pp_t \Phi_e\right|
\lesssim
\1_{\{ r\le \frac{\sqrt{t}}{2}\}}
\mu_0
\begin{cases}
	t^{-1-\frac{\gamma}{2}} \ln t
	\langle \bar\rho  \rangle^{-1}
	\ln (\bar\rho+2)
	,
	&
	1<\gamma<2
	\\
	t^{-2} \ln t \langle \bar\rho  \rangle^{-1}
	\ln (\bar\rho+2)  ,
	& \gamma=2
	\\
	t^{-2}  \langle \bar\rho  \rangle^{-1}
	\ln (\bar\rho+2) ,
	&
	\gamma>2 ,
\end{cases}
\end{aligned}
\end{equation*}

\noindent $\bullet$
\begin{align*}
&~\left| \left( \eta(4z)  -
\eta(z) \right) \bar{\mu}_0^{-1}
\frac{-8 \bar\rho }{(\bar{\rho}^2+1)^2}
\left( \varphi[\bar{\mu}_0](r,t)
+
\psi_*(r,t)
\right) \right|
\\
\lesssim  &~
\1_{\{ \frac{t^{\frac{1}{2}}}{4} \le r \le 2t^{\frac{1}{2}}\}}
\mu_0^2 t^{-\frac{3}{2}}
\begin{cases}
t^{-\frac{\gamma}{2}} ,
& 1<\gamma<2
\\
t^{-1},
& \gamma=2
\\
(t \ln t)^{-1} ,
& \gamma >2 ,
\end{cases}
\end{align*}
and
\begin{equation*}
\begin{aligned}
&~\left|\Phi_e \eta'(4z)\frac{2r}{t^{\frac{3}{2}}}+\frac{16}{t}\eta''(4z)\Phi_e+\frac{8}{\sqrt t}\eta'(4z)\pp_r \Phi_e
+  \frac1r \frac4{\sqrt t}\eta'(4z)\Phi_e\right|\\
\lesssim &~ \1_{\{\frac{t^{\frac 12}}{4} \le r\le \frac{t^{\frac12}}{2} \}}\begin{cases}
\mu_0^2 t^{-\frac{3+\gamma}{2}}\ln t ,\quad & 1<\gamma<2\\
\mu_0^2 t^{-\frac52}\ln t ,\quad & \gamma= 2\\
\mu_0^2 t^{-\frac52},\quad & \gamma> 2\\
\end{cases}
\end{aligned}
\end{equation*}
by using \eqref{Phie-est1}.

\noindent $\bullet$
\begin{equation*}
\begin{aligned}
	&
	\left|
\eta(z)
\left(
\mu^{-1}  \frac{8 \rho }{ \left( \rho^{2}+1\right)^2}
-
\bar{\mu}_0^{-1}
\frac{8 \bar\rho }{(\bar{\rho}^2+1)^2}
\right)
\left( \varphi[\bar{\mu}_0](r,t)
+
\psi_*(r,t)
\right)
\right|
\\
\lesssim \ &
\eta(z)
|\mu_1| \mu_0^{-2}
\langle \rho \rangle^{-3}
\begin{cases}
	t^{-\frac{\gamma}{2}} ,
	& 1<\gamma<2
	\\
	t^{-1},
	& \gamma=2
	\\
	(t \ln t)^{-1} ,
	& \gamma >2 ,
\end{cases}
\end{aligned}
\end{equation*}

\noindent $\bullet$
\begin{equation*}
	\begin{aligned}
		&
		\left|\eta(z) \mu^{-1}  \frac{8 \rho }{ \left( \rho^{2}+1\right)^2}
		\left( \varphi[\mu](r,t) -
		\varphi[\bar{\mu}_0](r,t)
		\right)\right|
		\\
		\lesssim \ &
		\eta(z) \mu^{-1}  \frac{ \rho }{ \left( \rho^{2}+1\right)^2}
		\Bigg[
		\left( O ( |\mu_1|  t^{-1} )
		+
		\tilde{g}[\mu,\mu_1] \right) \1_{\{ r  \le 2t^{\frac 12} \}}
		\\
		&
		+
		\sup\limits_{t_1\in[t/2,t]} \Big(\frac{|\mu_1(t_1)|}{ \mu(t) } + \frac{|\dot\mu_{1}(t_1)|}{|\dot\mu(t)|}  \Big)
		\begin{cases}
			|\dot \mu| \langle \ln (\mu^{-1} t^{\frac 12} ) \rangle
			&
			\mbox{ \ if \ } r\le \mu
			\\
			|\dot \mu| \langle\ln(r^{-1} t^{\frac 12}) \rangle
			&
			\mbox{ \ if \ }  \mu < r\le t^{\frac 12}
			\\
			t |\dot \mu| r^{-2} e^{-\frac{r^2}{16 t}}
			&
			\mbox{ \ if \ }   r > t^{\frac 12}
		\end{cases}
		\Bigg] .
	\end{aligned}
\end{equation*}
$\bullet$
\begin{equation*}
\left|\eta(4z) \bar{\mu}_0^{-1}
\mathcal{M}[\bar{\mu}_0]
\frac{\eta(\bar{\rho}) \mathcal{Z}(\bar{\rho})  }{\int_{0}^3 \eta(x) \mathcal{Z}^2(x) xdx } \right|
\lesssim
\mu_0^{-1} t^{-2} \eta(\bar{\rho}) \bar{\rho} .
\end{equation*}
$\bullet$
\begin{equation*}
	\left|
\eta(4z) \bar{\mu}_0^{-2}
\left(
-
\frac{\cos(2Q_\mu)-\cos(2Q_{\bar{\mu}_0})}{\bar{\rho}^2}
\Phi_e
\right) \right|
\lesssim
\eta(4z) \mu_1 \mu_0^{-3} |\Phi_e|
\langle \bar{\rho}\rangle^{-4} .
\end{equation*}

For the remaining error $E_e$, we have the following

\noindent $\bullet$ If $z<1$, then
\begin{equation}
\begin{aligned}
E_e=&~
- \frac{1}{2r^2}
\left[ \sin[2( Q_\mu+\Phi_1+\Phi_2+\Psi_{*} +\Phi_e )]
-
\sin(2Q_\mu)
-
\cos(2Q_\mu) 2\left( \Phi_1+\Phi_2+\Psi_{*} + \Phi_e \right)
\right]
\\
&~
+
\frac{1}{2r^2}
\Big[
- \sin[2(Q_\mu+\Phi_1+\Phi_2+\Psi_{*} + \Phi_e )]
+
 \sin[2( Q_\mu+\Phi_1+\Phi_2+\Psi_{*} +\Phi_e )]
\Big] \\
=&~\frac{1}{2r^2}
\Big(
 \sin[2( Q_\mu+\Phi_1+\Phi_2+\Psi_{*} +\Phi_e )]-\sin(2Q_\mu)-
2\cos(2Q_\mu) \left( \Phi_1+\Phi_2+\Psi_{*} + \Phi_e \right)
\Big)
\end{aligned}
\end{equation}
Since
$$
\Big|\Phi_1+\Phi_2+\Psi_{*} + \Phi_e\Big|\ll Q_\mu,
$$
by Taylor expansion, we have
\begin{align*}
|E_e|\lesssim&~\frac1{r^2}\Big|\sin(2Q_\mu)(\Phi_1+\Phi_2+\Psi_{*} + \Phi_e)^2\Big|\\
\lesssim&~\langle\rho\rangle^{-1}(\varphi+\psi_*)^2+\mu^{-2}\rho^{-1}\langle\rho\rangle^{-2}\Phi_e^2
\end{align*}

\noindent $\bullet$ If $1<z<2$, then
\begin{equation}
\begin{aligned}
E_e=&~
-
\eta(z) \frac{1}{2r^2}
\left[ \sin[2( Q_\mu+\Phi_1+\Phi_2+\Psi_{*} +  \Phi_e )]
-
\sin(2Q_\mu)
-
\cos(2Q_\mu) 2\left( \Phi_1+\Phi_2+\Psi_{*} + \Phi_e \right)
\right]
\\
&~
+
\frac{1}{2r^2}
\Big[
- \sin[2(\eta(z) Q_\mu+\Phi_1+\Phi_2+\Psi_{*} +  \Phi_e )]
+
\eta(z) \sin[2( Q_\mu+\Phi_1+\Phi_2+\Psi_{*} +  \Phi_e )]
\\
&~ + 2\left( 1-\eta(z)  \right) \left( \Phi_1+\Phi_2+\Psi_{*} \right)
\Big].
	\end{aligned}
\end{equation}
Since now $|Q_\mu|\ll 1$, we have
\begin{equation}
\begin{aligned}
|E_e|\lesssim \eta(z) \frac{1}{r^2} \Big|\sin(2Q_\mu)(\Phi_1+\Phi_2+\Psi_{*} + \Phi_e)^2\Big|+r^{-2}(|Q_\mu+\Phi_e|)\1_{\{r\sim \sqrt t\}}.
\end{aligned}
\end{equation}

\noindent $\bullet$ If $z>2$, then
\begin{equation}
\begin{aligned}
E_e=&~
\frac{1}{2r^2}
\Big[
- \sin[2(\Phi_1+\Phi_2+\Psi_{*} + \eta(4z) \Phi_e )]
 + 2\left( \Phi_1+\Phi_2+\Psi_{*} \right)
\Big],
	\end{aligned}
\end{equation}
and thus
\begin{align*}
|E_e|\lesssim t^{-1}|\Phi_e|.
\end{align*}

\noindent $\bullet$ We will need to take into account the cancellation in \eqref{varphi+psi} for the estimate of $\varphi[\mu] + \psi_*$. Since
\begin{equation}
\begin{aligned}
&~\left| \mu_1  t^{-1}
		+ \int_{t/2}^{t-\mu_{0}^2 } \frac{\dot\mu_1(s)}{t-s} d s
		\right|\\
\lesssim &~\vartheta+\dot\mu_1\ln t+[\dot\mu_1]_{C^\alpha}\max\{\mu_0^{2\alpha},t^{(1-\nu)\alpha}\}\\
\lesssim&~\vartheta\ln t+t^{1-\alpha}(\mu_0 R)^{-2}\vartheta\max\{\mu_0^{2\alpha},t^{(1-\nu)\alpha}\}\\
\lesssim&~\vartheta\ln t+t^{1-\alpha\nu}(\mu_0 R)^{-2}\vartheta\end{aligned}
\end{equation}
we have
\begin{equation*}
	\begin{aligned}
			\varphi[\mu] + \psi_*
		= ~&
		\bigg[ \vartheta\ln t+t^{1-\alpha\nu}(\mu_0 R)^{-2}\vartheta	+ O( \mu_0  t^{-2} r^2 )
		+
		|\dot{\mu}_0|
		\min\{ \langle \rho \rangle,\ln t \} \big)
		\bigg]
		\1_{\{ r\le 2t^{\frac 12}  \}}
		\\
		& +
		O\Big(  \mu_0  r^{-2} e^{-\frac{r^2}{16 t}}
		+
		|\dot{\mu}_0|  t^3 r^{-6}
		\Big)
		\1_{\{ r > 2 t^{\frac 12} \}}
		+ O( \mu_0 \rho  t^{-\frac{1}{2}} v_{\gamma}(t) )
		+
		O\left( \ln \ln t |\dot{\mu}_0| \right) .
	\end{aligned}
\end{equation*}

\medskip

\subsection{Estimates for the outer problem}\label{Sec-outer}

Recall the norm of the outer problem defined in \eqref{norm-out}. We will solve \eqref{out-eq} in the space
\begin{equation*}
B_{\rm{out}} : = \left\{ f: ~ \|f\|_{{\rm out}} \le C_{o} \right\}
\end{equation*}
for a large costant $C_{o}$.
For any $\psi \in B_{\rm{out}} $, we will estimate the right hand side of the outer problem, $\mathcal G$ defined in \eqref{def-G}.

\noindent $\bullet$ By above estimates, we have
\begin{equation}
\begin{aligned}
&~\left|(1-\eta_R)E\left[v_1 \right]\right|
\lesssim
\1_{\{  2\mu_0 R\leq r\leq 2\sqrt{t} \}}
\mu_0^2
\begin{cases}
	t^{-1-\frac{\gamma}{2}} (\ln t)^2
	\langle r \rangle^{-1}
	,
	&
	1<\gamma<2
	\\
	t^{-2} (\ln t)^2 \langle r  \rangle^{-1}
 ,
	& \gamma=2
	\\
	t^{-2} \ln t \langle r  \rangle^{-1}
	 ,
	&
	\gamma>2
\end{cases}
\quad
+\1_{\{2 \mu_0 R\leq r\leq 2\sqrt t \}}\mu_0^{2}\ln t \vartheta(t) \langle r\rangle^{-3} \| \mu_{1} \|_{*}
\\
&~+\1_{\{2 \mu_0 R\leq r\leq 2\sqrt t \}} \vartheta(t) \mu_0^{2}\langle r\rangle^{-4}\begin{cases}
		\min\left\{
		\mu_0 t^{1-\frac{\gamma}{2}}
	\langle r  \rangle^{-1}
	\ln t
		,
		t^{1-\frac{\gamma}{2}} (\ln t)^{-1} \ln( \ln t )
		\right\}
		,
		&
		1<\gamma<2
		\\
		\min\left\{  \mu_0
		 \langle r  \rangle^{-1}
		\ln t,
		 (\ln t)^{-2} \ln( \ln t )  \right\} ,
		& \gamma=2
		\\
		\min\left\{
		 \mu_0 \langle r  \rangle^{-1}
		 ,
		 (\ln t)^{-2} \ln( \ln t ) \right\},
		&
		\gamma>2
	\end{cases}  \\
&~+\1_{\{2 \mu_0 R\leq r\leq 8\sqrt t \}}\mu_0^3\langle r\rangle^{-3}
\begin{cases}
		\min\left\{
		t^{-\gamma} \mu_0^2
	\langle r \rangle^{-2}
	(\ln t)^2
		,
		t^{-\gamma} (\ln t)^{-2} (\ln( \ln t ) )^2
		\right\}
		,
		&
		1<\gamma<2
		\\
		\min\left\{
		t^{-2}\mu_0^2 \langle r  \rangle^{-2}
		(\ln t)^2,
		t^{-2} (\ln t)^{-4} (\ln( \ln t ))^2  \right\} ,
		& \gamma=2
		\\
		\min\left\{
		t^{-2} \mu_0^2 \langle r  \rangle^{-2}
		,
		t^{-2} (\ln t)^{-4} (\ln( \ln t ))^2 \right\},
		&
		\gamma>2
	\end{cases}  \\
&~+\1_{\{2 \mu_0 R\leq r\leq 2\sqrt t \}}\mu_0\langle r\rangle^{-1}\Big[ \vartheta(t)^2(\ln t)^2+t^{2-2\alpha\nu}(\mu_0 R)^{-4}\vartheta(t)^2+\mu_0^2 t^{-2}+(\dot\mu_0 \ln t)^2\Big],
\end{aligned}
\end{equation}
and for $r\leq \sqrt t$, we have
\begin{equation}
\begin{aligned}
&
T_4 \bullet
 \Big[r^{-1}(1-\eta_R)E\left[v_1 \right]\Big](x,t,t_0)
\lesssim
\Bigg\{
\mu_0^2
\begin{cases}
	t^{-1-\frac{\gamma}{2}} (\ln t)^3
	,
	&
	1<\gamma<2
	\\
	t^{-2} (\ln t)^3
 ,
	& \gamma=2
	\\
	t^{-2} (\ln t)^2
	 ,
	&
	\gamma>2
\end{cases}
\quad
+
\vartheta(t)  R^{-2}\ln t
\| \mu_{1} \|_{*}
\\
&~+
\vartheta(t) \mu_0^{-1}R^{-4}\begin{cases}
		t^{1-\frac{\gamma}{2}}
	\ln t
		,
		&
		1<\gamma<2
		\\
		\ln t,
		& \gamma=2
		\\
	1
		 ,
		&
		\gamma>2
	\end{cases}
	\quad
	+\mu_0^{-1}R^{-4}
\begin{cases}
		t^{-\gamma} \mu_0^2
	(\ln t)^2
		,
		&
		1<\gamma<2
		\\
		t^{-2}\mu_0^2
		(\ln t)^2,
		& \gamma=2
		\\
		t^{-2} \mu_0^2
	        ,
		&
		\gamma>2
	\end{cases}  \\
&~+\mu_0\ln t\Big[ \vartheta(t)^2(\ln t)^2+t^{2-2\alpha\nu}(\mu_0 R)^{-4} \vartheta(t)^2+\mu_0^2 t^{-2}+(\dot\mu_0 \ln t)^2\Big]
 \Bigg\} \left(\1_{\{ r\le t^{\frac{1}{2}} \}} + t r^{-2}  \1_{\{ r> t^{\frac{1}{2}} \}} \right)
 \\
 \lesssim \ &
 t^{-\epsilon}
  w_{o}(r,t)
.
\end{aligned}
\end{equation}

For the coupling terms, we have
\begin{equation}
	\begin{aligned}
		&~
		\Bigg|\eta''(\frac{\rho}{R}) (\mu R)^{-2}
		\phi
		+\eta'(\frac{\rho}{R}) \mu^{-2} (\rho R)^{-1} \phi
		+ 2 (\mu R)^{-1} \mu^{-1}
		\eta'(\frac{\rho}{R}) \pp_{\rho} \phi
		+
		\eta'(\frac{\rho}{R}) \frac{\rho}{R}
		\frac{(\mu R)'}{\mu R} \phi
		\\
		& ~
		-
		\mu^{-2} \left(\eta(z) -1 \right)
		\frac{  \rho^{4} - 6 \rho^{2} + 1}{ \rho^2\left( \rho^{2} + 1  \right)^2}  \eta_{R} \phi \Bigg|
		\lesssim
		\tau^{-\kappa}\mu_0^{-2}R^{-2-a}\1_{\{\mu R \le r \le 2\mu R \}}
	\end{aligned}
\end{equation}
whose contribution for $\psi$ is given by
\begin{equation*}
\begin{aligned}
&
T_4 \bullet \left[ r^{-1} \tau^{-\kappa}(t) \mu_0^{-2}R^{-2-a}\1_{\{\mu R \le r \le 2\mu R \}}\right]
\lesssim
T_4 \bullet  \left[   \tau^{-\kappa}(t) \mu_0^{-3}R^{-3-a}\1_{\{\mu R \le r \le 2\mu R \}}\right]
\\
\lesssim \ &
t^{-2} e^{-\frac{r^2}{16 t}}
\int_{\frac{t_0}{2}}^{\frac{t}{2}}
\tau^{-\kappa}(s) \mu_0(s) R^{1-a}(s) ds
+
\tau^{-\kappa}(t) \mu_0^{-1}(t)R^{-1-a}(t)
\left(\1_{\{ r\le 2\mu_0 R \}} + (\mu_0 R)^2 r^{-2} e^{-\frac{r^2}{16 t}} \1_{\{ r> 2\mu_0 R \}} \right)
\\
\lesssim \ &
  w_{o}(r,t)
\end{aligned}
\end{equation*}
provided
\begin{equation}\label{out-restrict}
\begin{cases}
2-\frac{\gamma}{2} -\kappa(\gamma-1) + \omega(1-a)>0 , &
1<\gamma<2
\\
1 -\kappa + \omega(1-a)>0 , &
\gamma\ge 2
\end{cases}
\end{equation}

\begin{equation}
\begin{aligned}
	\left|(1-\eta_R) \mu^{-2} \frac{8}{ \left( \rho^{2}+1\right)^2}  \Psi \right|
	\lesssim \1_{\{\mu_0 R/2 \le r\}}\mu_0^2 r^{-3} |\psi|
	\lesssim
	\1_{\{\mu_0 R/2 \le r\}}\mu_0 R^{-1} r^{-2} |\psi|
\end{aligned}
\end{equation}
Then this term  contributes to the outer problem with the form
\begin{equation*}
T_4\bullet  \left[ \1_{\{\mu_0 R/2 \le r\}}\mu_0 R^{-1} r^{-3} |\psi| \right]
\lesssim
T_4\bullet  \left[ \1_{\{\mu_0 R/2 \le r\}}  R^{-2} r^{-2} |\psi| \right]
\lesssim
R^{-2} w_{o}(r,t) .
\end{equation*}

\noindent $\bullet$ Estimate of the nonlinear terms defined by
\begin{equation}
\begin{aligned}
\mathcal N:=&~
\frac{1}{2r^2}
\Big[ \sin\left( 2\left( \eta(z) Q_\mu+\Phi_1+\Phi_2+\Psi_{*} +\eta(4z)\Phi_e \right) \right)
-
\eta(z) \sin\left( 2\left(  Q_\mu+\Phi_1+\Phi_2+\Psi_{*} +\eta(4z)\Phi_e \right) \right)
\\
&~+
\eta(z) \sin\left( 2\left(  Q_\mu+\Phi_1+\Phi_2+\Psi_{*} +\eta(4z)\Phi_e + \Psi  + \eta_{R} \phi \right) \right)
\\
&~
-\sin\left( 2\left( \eta(z) Q_\mu+\Phi_1+\Phi_2+\Psi_{*} +\eta(4z)\Phi_e + \Psi  + \eta_{R} \phi \right) \right) \Big]
\\
&~ +
\frac{\eta(z)}{2r^2}
\Big\{ \sin\left( 2\left(  Q_\mu+\Phi_1+\Phi_2+\Psi_{*} +\eta(4z)\Phi_e \right) \right)
-\sin(2Q_{\mu})
-2\cos(2Q_{\mu}) \left(  \Phi_1+\Phi_2+\Psi_{*} +\eta(4z)\Phi_e \right)
\\
&~
-
\Big[
\sin\left( 2\left(  Q_\mu+\Phi_1+\Phi_2+\Psi_{*} +\eta(4z)\Phi_e + \Psi  + \eta_{R} \phi \right) \right)
-\sin(2Q_{\mu})
\\
&~ -2\cos(2Q_{\mu}) \left( \Phi_1+\Phi_2+\Psi_{*} +\eta(4z)\Phi_e + \Psi  + \eta_{R} \phi \right)
\Big]
\Big\}- \frac{\eta(z)-1}{r^2} \cos(2Q_{\mu}) \Psi \\
=&~
\frac{1}{2r^2}
\Big[ \sin\left( 2v_1\right)
-\sin\left( 2\left( v_1 + \Psi  + \eta_{R} \phi \right) \right)+2\eta(z)\cos(2Q_{\mu}) \left(\eta_{R} \phi \right)
 +2\cos(2Q_{\mu}) \Psi \Big]
\end{aligned}
\end{equation}
\noindent $\bullet$ If $0<z\le 1$, then
\begin{equation*}
	\begin{aligned}
		|\mathcal N|\lesssim&~\frac1{r^2}\sin(2Q_\mu) \left( \Phi_1+\Phi_2+\Psi_{*} +\eta(4z)\Phi_e + \Psi  + \eta_{R} \phi \right)^2\\
		\lesssim&~\1_{\{r\leq \sqrt t\}} r^{-2}\rho\langle \rho\rangle^{-2}\left[r^2\left(\varphi[\mu] + \psi_* \right)^2
		+
r^2 \psi^2
		+\Phi_e^2+(\eta_{R} \phi)^2\right]
	\end{aligned}
\end{equation*}
Therein,
\begin{equation*}
\begin{aligned}
&
\1_{\{r\leq \sqrt t\}}  \rho\langle \rho\rangle^{-2}
 \psi^2
 \lesssim
\tau^{-\kappa}(t) \mu_0^{-1}(t)R^{-1-a}(t)
 \left( \1_{\{ r\le \mu_0 \}}
 +
 \1_{\{ \mu_0 < r\leq 2 \sqrt t\}}  \frac{\mu_0}{r}
 \right)
 \psi \|\psi\|_{{\rm out}}
\\
\lesssim \ &
\tau^{-\kappa}(t) \mu_0^{-1}(t)R^{-1-a}(t)
\left( \1_{\{ r\le 1 \}}
+
\1_{\{ 1 < r\leq 2 \sqrt t\}}  \frac{\mu_0}{r}
\right)
\psi \|\psi\|_{{\rm out}} .
\end{aligned}
\end{equation*}
Then
\begin{equation*}
T_4 \bullet [r^{-1} \1_{\{r\leq \sqrt t\}}  \rho\langle \rho\rangle^{-2}
\psi^2 ] \lesssim t^{-\epsilon}   w_{o}(r,t) \|\psi\|_{{\rm out}}^2 .
\end{equation*}

\begin{equation*}
\1_{\{r\leq \sqrt t\}}  \rho\langle \rho\rangle^{-2}  \left(\varphi[\mu] + \psi_* \right)^2
\lesssim
\1_{\{r\leq 2 \sqrt t\}}   \langle \rho\rangle^{-1}  \left(\varphi[\mu] + \psi_* \right)^2
\end{equation*}
which implies
\begin{equation*}
\begin{aligned}
&
	\left|
T_4 \bullet [ r^{-1} \1_{\{r\leq 2 \sqrt t\}}   \langle \rho\rangle^{-1}  \left(\varphi[\mu] + \psi_* \right)^2 ] \right|
\\
\lesssim \ &
\left(
\tau^{-\kappa}(t) \mu_0^{-1}(t)R^{-1-a}(t) \ln t+t^{1-\alpha\nu}(\mu_0 R)^{-2} \tau^{-\kappa}(t) \mu_0^{-1}(t)R^{-1-a}(t)	+ O( \mu_0  t^{-1} )
+
|\dot{\mu}_0|  \ln t
\right)^2
\\
& \times
\left(\1_{\{ r\le t^{\frac{1}{2}} \}} + t r^{-2}  \1_{\{ r> t^{\frac{1}{2}} \}} \right)
\lesssim
t^{-\epsilon}  w_{o}(r,t)  ,
\end{aligned}
\end{equation*}
\begin{equation*}
\left|T_4 \bullet  \left[r^{-1}\1_{\{r\leq \sqrt t\}} r^{-2}\rho\langle \rho\rangle^{-2}\big[\Phi_e^2+(\eta_{R} \phi)^2\big]\right]\right|
\lesssim
\begin{cases}
		t^{-\gamma} \mu_0^2
	(\ln t)^2
		,
		&
		1<\gamma<2
		\\
		t^{-2}\mu_0^2
		(\ln t)^2,
		& \gamma=2
		\\
		t^{-2} \mu_0^2 ,
		&
		\gamma>2
	\end{cases} \lesssim
t^{-\epsilon}  w_{o}(r,t)  .
\end{equation*}

\noindent $\bullet$ If $1<z\le 2$, then
\begin{equation}
	\begin{aligned}
		|\mathcal N|
		\lesssim&~\1_{\{ t^{\frac{1}{2}} \le r \le 2t^{\frac{1}{2}} \}} r^{-2}\rho^{-2}|\Psi|
		+
		\1_{\{ t^{\frac{1}{2}} \le r \le 2t^{\frac{1}{2}} \}} r^{-2}\rho\langle \rho\rangle^{-2}\Big[r^2(\varphi[\mu] + \psi_*+\psi)^2+\Phi_e^2\Big]
	\end{aligned}
\end{equation}
Therein,
\begin{equation*}
T_4 \bullet \left[ \1_{\{ t^{\frac{1}{2}} \le r \le 2t^{\frac{1}{2}} \}} r^{-3}\rho^{-2}|\Psi| \right]
\lesssim
T_4  \bullet  \left[ \1_{\{ t^{\frac{1}{2}} \le r \le 2t^{\frac{1}{2}} \}} r^{-4} \mu_0^2 |\psi|  \right]
\lesssim
T_4 \bullet  \left[ \1_{\{ t^{\frac{1}{2}} \le r \le 2t^{\frac{1}{2}} \}} t^{-\epsilon} r^{-2}  |\psi|  \right]
\lesssim
t^{-\epsilon}  w_{o}(r,t)
\end{equation*}
for an $\epsilon>0$ sufficiently small since $\mu_0\ll t^{\frac 12 -}$.

\noindent $\bullet$ If $z>2$, then
\begin{equation*}
	\begin{aligned}
		|\mathcal N|
		\lesssim \1_{\{r\geq  2\sqrt t\}} r^{-2}\rho^{-2}|\Psi|
\lesssim
\1_{\{r\geq  2\sqrt t\}} r^{-3} \mu_0^2 |\psi| .
	\end{aligned}
\end{equation*}
Then
\begin{equation*}
T_4 \bullet   \left[  \1_{\{r\geq  2\sqrt t\}} r^{-4} \mu_0^2 |\psi|  \right]
\lesssim
T_4 \bullet  \left[  \1_{\{r\geq  2\sqrt t\}} t^{-\epsilon} r^{-2}  |\psi|  \right]
\lesssim
t^{-\epsilon}  w_{o}(r,t) .
\end{equation*}
From the above estimate, choosing $C_0$ large, then making $R, t$ large enough, we have  $T_4 \bullet \mathcal{G} \in B_{\rm{out}}$. The contraction mapping property can be derived very similarly. Thus we can find the solution of the outer problem in $B_{\rm{out}}$.

\medskip

\subsection{Estimates for the inner problem}

In this section, we estimate
\begin{equation}
\frac{8}{ \left( \rho^{2}+1\right)^2} \Psi
		+
		\mu^2  E\left[v_1 \right]
\end{equation}

From the beginning of this section, we have the following estimates in the inner region $(\rho,t)\in \mathcal D_{2R}$.

\noindent $\bullet$

\begin{equation}
\begin{aligned}
\left|\frac{8}{ \left( \rho^{2}+1\right)^2} \Psi\right|\lesssim &~v_{{\rm out}}(t)(\mu_0 R)^{-1-b}\langle \rho\rangle^{-3}\1_{\{\rho\leq 2R\}}\\
=&~\mu^{a}\tau^{-\kappa}(\mu_0 R)^{-1-b}\langle \rho\rangle^{-3}\1_{\{\rho\leq 2R\}},
\end{aligned}
\end{equation}

\noindent $\bullet$
\begin{equation*}
\begin{aligned}
\left| \eta(4z)\pp_t \Phi_e\right|
\lesssim
\1_{\{ r\le \frac{\sqrt{t}}{2}\}}
\mu_0
\begin{cases}
	t^{-1-\frac{\gamma}{2}} \ln t
	\langle \bar\rho  \rangle^{-1}
	\ln (\bar\rho+2)
	,
	&
	1<\gamma<2
	\\
	t^{-2} \ln t \langle \bar\rho  \rangle^{-1}
	\ln (\bar\rho+2)  ,
	& \gamma=2
	\\
	t^{-2}  \langle \bar\rho  \rangle^{-1}
	\ln (\bar\rho+2) ,
	&
	\gamma>2 ,
\end{cases}
\end{aligned}
\end{equation*}

\noindent $\bullet$
\begin{equation*}
\begin{aligned}
	&
	\left|
\eta(z)
\left(
\mu^{-1}  \frac{8 \rho }{ \left( \rho^{2}+1\right)^2}
-
\bar{\mu}_0^{-1}
\frac{8 \bar\rho }{(\bar{\rho}^2+1)^2}
\right)
\left( \varphi[\bar{\mu}_0](r,t)
+
\psi_*(r,t)
\right)
\right|
\\
\lesssim \ &
\eta(z)
t\vartheta \mu_0^{-2}
\langle \rho \rangle^{-3}
\begin{cases}
	t^{-\frac{\gamma}{2}} ,
	& 1<\gamma<2
	\\
	t^{-1},
	& \gamma=2
	\\
	(t \ln t)^{-1} ,
	& \gamma >2 .
\end{cases}
\end{aligned}
\end{equation*}
Recall that $\tilde{g}[\mu,\mu_1]$ is defined in Proposition \ref{varphi-coro}. By
\begin{equation}
\begin{aligned}
\frac{|\mu_1(t_1)|}{ \mu(t) } + \frac{|\dot\mu_{1}(t_1)|}{|\dot\mu(t)|}  \lesssim &~\frac{t\vartheta} {\mu_0}+\frac{\vartheta}{|\dot\mu_0|},
\end{aligned}
\end{equation}
we have
\begin{equation}
\begin{aligned}
|\tilde{g}[\mu,\mu_1]|\lesssim &~ |\dot\mu_0|\ln t\left(\frac{t\vartheta} {\mu_0}+\frac{\vartheta}{|\dot\mu_0|}\right)^2+|\dot\mu_0|\left(\frac{t\vartheta} {\mu_0}+\frac{\vartheta}{|\dot\mu_0|}\right)\\
&~+t^{-2} \int_{t_0/2}^{t}
			\Big[ s^{-1} \vartheta(s) \mu_0^{2}(s) + s |\dot\mu(s)| \Big(\frac{t\vartheta} {\mu_0}+\frac{\vartheta}{|\dot\mu_0|}  \Big) \Big]  ds\\
\lesssim&~ \vartheta.
\end{aligned}
\end{equation}
So one has
\begin{equation*}
	\begin{aligned}
		\left|\eta(z) \mu^{-1}  \frac{8 \rho }{ \left( \rho^{2}+1\right)^2}
		\left( \varphi[\mu](r,t) -
		\varphi[\bar{\mu}_0](r,t)
		\right)\right|
		\lesssim  \eta(z) \mu_0^{-1} \vartheta  \langle \rho\rangle^{-3}
	\end{aligned}
\end{equation*}
$\bullet$
\begin{equation*}
\left|\eta(4z) \bar{\mu}_0^{-1}
\mathcal{M}[\bar{\mu}_0]
\frac{\eta(\bar{\rho}) \mathcal{Z}(\bar{\rho}) \rho }{\int_{0}^3 \eta(x) \mathcal{Z}^2(x) xdx }\right|
\lesssim
\mu_0^{-1} t^{-2} \eta(\bar{\rho}) \bar{\rho}^2 .
\end{equation*}
$\bullet$
\begin{align*}
&~	\left|
\eta(4z) \bar{\mu}_0^{-2}
\left(
-
\frac{\cos(2Q_\mu)-\cos(2Q_{\bar{\mu}_0})}{\bar{\rho}^2}
\Phi_e
\right) \right|\\
\lesssim &~\eta(4z) t\vartheta \mu_0^{-2}
\langle \bar{\rho}\rangle^{-4}\begin{cases}
		\min\left\{
		t^{-\frac{\gamma}{2}}
	\langle \bar\rho  \rangle^{-1}
	\ln (\bar\rho+2)
		,
		t^{-\frac{\gamma}{2}} (\ln t)^{-1} \ln( \ln t )
		\right\}
		,
		&
		1<\gamma<2
		\\
		\min\left\{
		t^{-1} \langle \bar\rho  \rangle^{-1}
		\ln (\bar\rho+2),
		t^{-1} (\ln t)^{-2} \ln( \ln t )  \right\} ,
		& \gamma=2
		\\
		\min\left\{
		t^{-1} (\ln t)^{-1} \langle \bar\rho  \rangle^{-1}
		\ln (\bar\rho+2) ,
		t^{-1} (\ln t)^{-2} \ln( \ln t ) \right\},
		&
		\gamma>2 .
	\end{cases}
\end{align*}

\noindent $\bullet$ Since $\Phi_e$ has vanishing at the origin, we have
\begin{align*}
|E_e|\lesssim&~\langle\rho\rangle^{-1}(\varphi+\psi_*)^2+\mu^{-2}\rho^{-1}\langle\rho\rangle^{-2}\Phi_e^2\\
\lesssim&~\langle\rho\rangle^{-1}\Big[\vartheta^2(\ln t)^2+t^{2-2\alpha\nu}(\mu_0 R)^{-4}\vartheta^2+\mu_0^6 R^4 t^{-4}+(\dot\mu_0 \ln t)^2\Big]\\
&~+\langle\rho\rangle^{-3}
\begin{cases}
		\min\left\{
		t^{-\gamma}
	\langle \bar\rho  \rangle^{-2}
	(\ln (\bar\rho+2) )^2
		,
		t^{-\gamma} (\ln t)^{-2} (\ln( \ln t ) )^2
		\right\}
		,
		&
		1<\gamma<2
		\\
		\min\left\{
		t^{-2} \langle \bar\rho  \rangle^{-2}
		(\ln (\bar\rho+2))^2,
		t^{-2} (\ln t)^{-4} (\ln( \ln t ))^2  \right\} ,
		& \gamma=2
		\\
		\min\left\{
		t^{-2} (\ln t)^{-2} \langle \bar\rho  \rangle^{-2}
		(\ln (\bar\rho+2) )^2,
		t^{-2} (\ln t)^{-4} (\ln( \ln t ))^2 \right\},
		&
		\gamma>2 .
	\end{cases}
\end{align*}

For above terms whose $\|\cdot\|_{v,\ell}$-norm to be bounded, we require
\begin{equation}
\begin{cases}
\begin{cases}
\mu_0^3 v^{-1}t^{-\frac{2+\gamma}{2}}R^{\ell-1}\ll1, &\quad 1<\gamma<2\\
\mu_0^3 v^{-1}t^{-2}R^{\ell-1}\ll1, &\quad \gamma\geq 2\\
\end{cases}\\
\begin{cases}
t^{\frac{2-\gamma}{2}}\vartheta v^{-1} \ll1, &\quad 1<\gamma<2\\
\vartheta v^{-1}\ll1, &\quad \gamma\geq 2\\
\end{cases}\\
\mu_0t^{-2} v^{-1}\ll 1, \quad \gamma>1\\
\mu_0^2 R^{\ell-1}v^{-1}\Big[\vartheta^2(\ln t)^2+t^{2-2\alpha\nu}(\mu_0 R)^{-4}\vartheta^2+\mu_0^6 R^4 t^{-4}+(\dot\mu_0 \ln t)^2\Big]\ll 1, \quad \gamma>1\\
\begin{cases}
\mu_0^2 v^{-1} t^{-\gamma} \ll1, &\quad 1<\gamma<2\\
\mu_0^2 v^{-1} t^{-2}\ll1, &\quad \gamma\geq 2\\
\end{cases}\\
\end{cases}
\end{equation}
Recall that $\tau(t)$ is defined in \eqref{def-tau} and
\begin{align*}
\vartheta=\mu_0^a (\mu_0 R)^{-1-a} \tau^{-\kappa}, \quad v^{-1}=\tau^{\kappa}R_0^{5}.
\end{align*}
Then we need
\begin{equation}\label{restriction2}
	\begin{cases}
		\begin{cases}
			2-2\gamma+\kappa(\gamma-1)+\omega(\ell-1)<0, &\quad 1<\gamma<2\\
			\kappa-2+\omega(\ell-1)<0, &\quad \gamma\geq 2\\
		\end{cases}\\
		\begin{cases}
			\omega(\ell-3-2a)-\kappa<0,\\
			\omega(\ell-7-2a)+2-2\alpha \nu-\kappa<0,\\
		\end{cases} \qquad \qquad  \gamma\geq 2\\
		\begin{cases}
			\omega(\ell-3-2a)-\kappa(\gamma-1) <0,\\
			\omega(\ell-7-2a)+2-2\alpha \nu-\kappa(\gamma-1)-2(2-\gamma)<0,\\
			4(1-\gamma)+\omega(\ell+3)+\kappa(\gamma-1)<0,\\
		\end{cases}\qquad \qquad 1<\gamma< 2\\
	\end{cases}
\end{equation}
for the inner problem.

\medskip

\medskip

\section{Linear theory for the inner problem}\label{sec-inner}

\medskip

In this section, we develop a linear theory for the inner problem. We consider
\begin{equation}
\left\{
\begin{aligned}
&\pp_\tau \phi=\mathcal L\phi+f_1\phi+f_2\rho\pp_\rho\phi+h, \qquad (\rho,\tau)\in\mathcal{D}_{R},\\
&\phi(\rho,\tau_0)=0, \qquad \rho\in[0,R(\tau_0)],\\
\end{aligned}
\right.
\end{equation}
where
\begin{align*}
&\tau(t)=\int_{t_0}^t \mu^{-2}(s)ds+\tau_0, \\
&\mathcal L:=\pp_{\rho\rho}+\frac1{\rho}\pp_{\rho}-\frac{\rho^4+1-6\rho^2}{\rho^2(\rho^2+1)^2}, \quad V:=\frac8{(\rho^2+1)^2},\\
&\mathcal{D}_{R} = \{
(\rho,\tau) \ : \
\rho\in [0,R(\tau)], \ \tau \in (\tau_0,\infty)
\},\\
&|h(\rho,\tau)|\lesssim v(\tau)\langle \rho\rangle^{-\ell},
\end{align*}
and we make the following assumptions
\begin{equation}\label{assumptions}
\begin{aligned}
&|f_1(\rho,\tau)|+|f_2(\rho,\tau)|+\rho |\pp_\rho f_2(\rho,\tau)|\leq C_f \tau^{-d}, \quad d>0,~C_f\geq0,\\
&R(\tau), ~v(\tau) \in C^{1}(\tau_0,\infty),\quad v(\tau)>0, \quad
1\ll R(\tau) \ll \tau^{\frac 12} ,\\
&v(\tau) = a_0 \tau^{a_1}(\ln \tau)^{a_2} (\ln\ln \tau)^{a_3} \cdots,\\
&R(\tau) = b_0 \tau^{b_1}(\ln \tau)^{b_2} (\ln\ln \tau)^{b_3} \cdots,\\
&v'(\tau)= O( \tau^{-1} v(\tau) ),\quad R'(\tau)= O(\tau^{-1}R(\tau)),
\end{aligned}
\end{equation}
where $a_0,b_0>0$, $a_{i}, b_{i}\in \RR$, $i=1,2,\dots$. We shall write $v=v(\tau),~R=R(\tau)$ for simplicity. Recall that the linearized operator $\mathcal L$ has kernels
\begin{equation}\label{kernels}
\mathcal Z(\rho)=\frac{\rho}{\rho^2+1},\quad \tilde{\mathcal Z}(\rho)=\frac{\rho^4+4\rho^2\ln\rho-1}{2\rho(\rho^2+1)}.
\end{equation}
Our aim is to find well-behaved $\phi$ for RHS $h$ in the weighted space with norm
\begin{equation*}
	\|h\|_{v, \ell} := \sup_{(\rho,\tau)\in \mathcal D_{ R}} v^{-1}(\tau) \langle \rho \rangle ^{\ell} |h(\rho,\tau)|
\end{equation*}
for some $1<\ell<3$. We have the following
\begin{prop}\label{linear-theory}
Consider
\begin{equation*}
	\begin{cases}
	\pp_{\tau} \phi =	\mathcal{L}  \phi+f_1\phi+f_2\rho\pp_\rho\phi + h(\rho,\tau)
	 \mbox{ \ in \ } \mathcal{D}_R,
		\\
\phi(\cdot,\tau_0) = 0
\mbox{ \ in \ } [0,R(\tau_0)],
	\end{cases}
\end{equation*}
and assume $\tau^d\gg R^2\ln R$.
If $\|h\|_{v,\ell}<+\infty$, then there exists a solution with
\begin{equation}\label{est-444444}
|\phi(\rho,\tau)|\lesssim R^2 \ln R v(\tau)\langle \rho\rangle^{-1}\|h\|_{v,\ell}.
\end{equation}
If in addition the orthogonality condition
\begin{equation}\label{orthog-454545}
\int_0^R h(\rho,\tau) \mathcal Z(\rho) \rho d \rho =0
\end{equation}
holds for all $\tau>\tau_0$, then there exists a solution satisfying
$$
|\phi(\rho,\tau)|\lesssim v(\tau)\|h\|_{v,\ell}\Big[ R^{5-\ell} \ln R \langle \rho\rangle^{-3}+\tau^{-d}v(\tau) R^{7-\ell} (\ln R)^2 \langle \rho\rangle^{-1}\Big].
$$
\end{prop}
\begin{proof}
We first show the linear estimates without orthogonality condition. We look for solution to
\begin{equation}\label{eqn-444444}
	\begin{cases}
	\pp_{\tau} \phi =	\mathcal{L}  \phi+f_1\phi+f_2\rho\pp_\rho\phi + h(\rho,\tau)
	 \mbox{ \ in \ } \mathcal{D}_R,
		\\
 \phi = 0
 \mbox{ \ on \ } \pp\mathcal{D}_R,
\quad
\phi(\cdot,\tau_0) = 0
\mbox{ \ in \ } [0,R(\tau_0)],
	\end{cases}
\end{equation}
where
$$
\pp \mathcal{D}_{R} = \{
(\rho,\tau) \ : \
\rho=R(\tau), \ \tau \in (\tau_0,\infty)
\} .
$$
We use the notation
$$
\|f\|_{L^2(B_R)}^2:=\int_0^R f^2(\rho)\rho d\rho,\quad Q_R(f,f):=-\int_0^{2R} \mathcal L(\phi) \phi\rho d\rho,
$$ and test above equation with $\rho\phi$ to get
\begin{equation}\label{eqn-454545}
\begin{aligned}
\frac12\pp_\tau\|\phi\|_{L^2(B_R)}^2+Q_R(\phi,\phi)=&~\int_0^R f_1\phi^2 \rho d\rho-\int_0^R \rho f_2 \phi^2-\frac12 \int_0^R \rho^2 \phi^2 \pp_\rho f_2+\int_0^R h\phi\rho d\rho\\
\leq &~\frac52 C_f \tau^{-d}\|\phi\|_{L^2(B_R)}^2+\|\phi\|_{L^2(B_R)}\|h\|_{L^2(B_R)}.
\end{aligned}
\end{equation}
By a coercive estimate in \cite[Lemma 9.2]{LLG}
$$
Q_R(\phi,\phi)\gtrsim \frac1{R^2\ln R} \|\phi\|_{L^2(B_R)}^2,
$$
one has
$$
\pp_\tau\|\phi\|_{L^2(B_R)}^2+\frac1{R^2\ln R}\|\phi\|_{L^2(B_R)}^2\lesssim R^2 \ln R \|h\|_{L^2(B_R)}^2
$$
provided
$$
\tau^d \gg R^2\ln R.
$$
Then Gr\"onwall's inequality yields
$$
\|\phi\|_{L^2(B_R)}\lesssim R^2 \ln R \|h\|_{L^2(B_R)}\lesssim R^2\ln R v(\tau)\|h\|_{v,\ell}
$$
To get the pointwise control, we introduce the energy norm
$$
\|f\|_{X(B_R)}^2 \int_0^R \left((\pp_\rho f)^2+\frac{f^2}{\rho^2}\right) \rho d\rho,
$$
and the following embedding holds (cf. \cite[page 216]{Gustafson10})
\begin{equation}\label{embedding}
\|f\|^2_{L^{\infty}(B_R)}\leq \|f\|_{X(B_R)}^2.
\end{equation}
Integrating both sides of \eqref{eqn-454545} implies
$$
\int_\tau^{\tau+1} Q_R(\phi,\phi)\lesssim R^4(\ln R)^2 v^2(\tau)\|h\|^2_{v,\ell},
$$
and thus
\begin{equation}\label{eqn-464646}
Q_R(\phi,\phi)(\tilde\tau)\lesssim R^4(\ln R)^2 v^2(\tau)\|h\|^2_{v,\ell}
\end{equation}
for some $\tilde\tau\in(\tau,\tau+1)$. Next we multiply equation \eqref{eqn-444444} by $\rho\mathcal L\phi$ and integrate by parts
$$
-\frac12\pp_{\tau}Q_R(\phi,\phi)\lesssim \|\mathcal L \phi\|^2_{L^2(B_R)}+\tau^{-2d}\|\phi\|_{L^2(B_R)}^2+\int_0^R h\mathcal L\phi \rho d\rho.
$$
Then using Young's inequality we get
$$
\pp_\tau Q_R(\phi,\phi)\lesssim v^2(\tau)\|h\|^2_{v,\ell}
$$
since $\tau^d \gg R^2\ln R$. By above inequality and \eqref{eqn-464646}, we obtain
$$
Q_R(\phi,\phi)(\tau+1)\lesssim R^4(\ln R)^2 v^2(\tau)\|h\|^2_{v,\ell}.
$$
By the arbitrariness of $\tau$ here and the initial condition $\phi(\cdot,\tau_0) = 0$ as well as the embedding \eqref{embedding}, we have
\begin{equation}\label{eqn-484848}
\|\phi\|_{L^{\infty}(B_R)}\lesssim \|\phi\|_{X(B_R)}\lesssim [Q_R(\phi,\phi)(\tau)]^\frac12+\|\phi\|_{L^2(B_R)}\lesssim R^2 \ln R v(\tau)\|h\|_{v,\ell}.
\end{equation}
Now we upgrade above pointwise control to estimate with spatial decay. We write equation \eqref{eqn-444444} as
\begin{equation*}
	\begin{cases}
	\pp_{\tau} \phi =	\left(\pp_{\rho\rho}+\frac1{\rho}\pp_\rho \phi-\frac1{\rho^2}\phi\right)+\tilde h
	 \mbox{ \ in \ } \mathcal{D}_R,
		\\
 \phi = 0
 \mbox{ \ on \ } \pp\mathcal{D}_R,
\quad
\phi(\cdot,\tau_0) = 0
\mbox{ \ in \ } [0,R(\tau_0)],
	\end{cases}
\end{equation*}
where
$$
\tilde h:=\frac{8\phi}{(\rho^2+1)^2}+f_1\phi+f_2\rho\pp_\rho\phi + h(\rho,\tau).
$$
So we have
\begin{align*}
|\phi|\lesssim&~\rho\left|\Gamma_4\bullet (\rho^{-1}|\tilde h|\1_{\rho\leq R(\tau)})\right|\lesssim R^2 \ln R v(\tau)\langle \rho\rangle^{-1}\|h\|_{v,\ell}.
\end{align*}
where $\Gamma_4$ is the heat kernel in $\R^4$, and we have used the fact $\tau^d\gg R^2 \ln R$ and the convolution estimates in \cite[Lemma A.2]{LLG}. The proof of \eqref{est-444444} is complete.

\medskip

Next, we handle the case with orthogonality condition. We first consider an elliptic problem
$$
\mathcal L H=h.
$$
By expressing
$$
H(\rho,\tau)=\tilde{\mathcal Z}(\rho)\int_0^{\rho}h(s,\tau)\mathcal Z(s) s ds-\mathcal Z(\rho)\int_0^\rho h(s,\tau)\tilde{\mathcal Z}(s) sds
$$
and using orthogonality \eqref{orthog-454545}, we get
\begin{equation}
\|H\|_{v,\ell-2}\lesssim \|h\|_{v,\ell}
\end{equation}
since $1<\ell<3$. We now consider
\begin{equation}\label{eqn-412412412}
	\begin{cases}
	\pp_{\tau} \Phi =	\mathcal{L}  \Phi+ H(\rho,\tau)
	 \mbox{ \ in \ } \mathcal{D}_{2R},
		\\
 \Phi = 0
 \mbox{ \ on \ } \pp\mathcal{D}_{2R},
\quad
\Phi(\cdot,\tau_0) = 0
\mbox{ \ in \ } [0,2R(\tau_0)].
	\end{cases}
\end{equation}
Similar to above process of getting non-orthogonal linear theory, we have
$$
|\Phi(\rho,\tau)|\lesssim v(\tau) R^{5-\ell} \ln R \langle \rho\rangle^{-1}\|H\|_{v,\ell-2}.
$$
Above pointwise estimate together with a scaling argument yield
$$
|\Phi|+\langle\rho\rangle|\pp_\rho \Phi|+\langle\rho\rangle^2|\pp_{\rho\rho} \Phi|\lesssim v(\tau) R^{5-\ell} \ln R \langle \rho\rangle^{-1}\|H\|_{v,\ell-2}.
$$
So we have
$$
\pp_\tau(\mathcal L\Phi)=\mathcal L(\mathcal L\Phi)+h
$$
with
$$
|\mathcal L\Phi|+\langle\rho\rangle|\pp_\rho (\mathcal L\Phi)|\lesssim v(\tau) R^{5-\ell} \ln R \langle \rho\rangle^{-3}\|H\|_{v,\ell-2}.
$$
We want to find a desired solution $\phi$ and consider the remainder
$$
\tilde\phi =\phi-\mathcal L\Phi
$$
which solves
$$
\pp_\tau \tilde\phi=\mathcal L\tilde\phi+f_1\tilde\phi+f_2 \rho\pp_\rho \tilde\phi+f_1\mathcal L\Phi+f_2\rho\pp_\rho(\mathcal L \Phi).
$$
By above non-orthogonal linear theory, we have the following control for $\tilde\phi$
$$
|\tilde\phi|\lesssim \tau^{-d}v(\tau) R^{7-\ell} (\ln R)^2 \langle \rho\rangle^{-1}\|h\|_{v,\ell},
$$
and thus
$$
|\phi|\lesssim v(\tau)\|h\|_{v,\ell}\Big[ R^{5-\ell} \ln R \langle \rho\rangle^{-3}+\tau^{-d}v(\tau) R^{7-\ell} (\ln R)^2 \langle \rho\rangle^{-1}\Big]
$$
as desired.
\end{proof}

Next we perform another re-gluing procedure to further improve the linear theory with orthogonality. We have
\begin{prop}
\label{R0-linear}
	Consider
	\begin{equation*}
		\begin{cases}
			\pp_{\tau} \phi=\mathcal L \phi  + f_1 \phi + f_2 \rho\pp_\rho \phi + h(\rho,\tau) +
			c(\tau) \eta(\rho) \mathcal Z(\rho)
			& \mbox{ in } \mathcal D_{R},
			\\
			\phi(\rho,\tau_0)=0 & \mbox{ in } [0,R(\tau_0)],
		\end{cases}
	\end{equation*}
where $\|h\|_{v,\ell}<\infty$ with $1<\ell<3$. Assume $\tau^{d}\gg \max\{R^2,~R_0^6\}$, $R_0=c_1\tau^{\delta}$ for some $\delta\ge 0$ and $c_1>0$, then for $\tau_0$ sufficiently large, there exists  $(\phi,c(\tau))$ solving above equation, and $(\phi,c)=(\TT_{3i}[h],c[h] )$ defines a linear mapping of $h$ with the estimates
	\begin{equation*}
		\langle \rho\rangle |\pp_\rho \phi| +	|\phi|
		\lesssim  R_0^{6-\ell}\ln R_0 v(\tau)
		\langle \rho \rangle^{-a} \|h\|_{v,\ell} , \quad a<\ell-2,
	\end{equation*}
	\begin{equation*}
		\begin{aligned}
			c[h](\tau)
			= \ &
			-\left(\int_0^2 \eta(\rho) \mathcal Z^2(\rho) \rho d\rho\right)^{-1} \left(\int_0^{2R_0}
			h(\rho,\tau) \mathcal Z(\rho)\rho d\rho
			+
			R_0^{-\epsilon_{0}}  O( v \|h \|_{v,\ell} )  \right)
		\end{aligned}
	\end{equation*}
for some $\epsilon_0>0$, and $O( v \|h \|_{v,\ell} )$ depends linearly on $h$.
\end{prop}
\begin{proof}
We decompose
$$\phi(\rho,\tau) = \eta_{R_0} \phi_{i}(\rho,\tau) + \phi_{o}(\rho,\tau),$$
where $\eta_{R_0} = \eta(\frac{\rho}{R_0})$. In order to find a solution $\phi$, it suffices to  find $(\phi_i,\phi_o)$ such that
	\begin{equation}
		\label{phio-eq}
		\begin{cases}
			\pp_{\tau} \phi_{o}
			=
			\pp_{\rho\rho}\phi_{o}+\frac1{\rho}\pp_\rho  \phi_{o}-\frac1{\rho^2}\phi_{o}
			+
			J[\phi_{o},\phi_{i}]
			\mbox{ \ in \ } \DD_{R},
			\\
			\phi_{o} = 0
			\mbox{ \ on \ } \pp \DD_{R},\quad
			\phi_{o} = 0
			\mbox{ \ in \ } [0,R(\tau_0)] ,
		\end{cases}
	\end{equation}
	\begin{equation}
		\label{phii-eq}
		\begin{cases}
			\pp_{\tau} \phi_{i}
			=
			\mathcal L\phi_i
			+ f_1 \phi_i
			+
			f_2 \rho\pp_\rho\phi_i
			+
			V\phi_o
			+ h
			+
			c(\tau) \eta(\rho) \mathcal Z(\rho)
			&
			\mbox{ \ in \ } \DD_{2R_0},
			\\
			\phi_{i} = 0
			&
			\mbox{ \ in \ } B_{2R(\tau_0)} ,
		\end{cases}
	\end{equation}
	where
	\begin{equation*}
		\begin{aligned}
			J[\phi_{o},\phi_{i}]= \ &
			f_1 \phi_o + f_2 \rho\pp_\rho\phi_o
			+
			(
			1-\eta_{R_0}
			)V \phi_o
			+
			A[\phi_{i}] + h(1-\eta_{R_0}) ,
			\\
			A[\phi_{i}] = \ &
			 \phi_i(\pp_{\rho\rho}\eta_{R_0}+\frac1{\rho}\pp_\rho \eta_{R_0}) + 2\pp_\rho\eta_{R_0} \pp_\rho\phi_i + f_2 \rho\pp_\rho\eta_{R_0} \phi_i -\pp_{\tau} \eta_{R_0} \phi_i ,
		\end{aligned}
	\end{equation*}
and
	\begin{equation*}
		c(\tau) = c[\phi_{o}](\tau) =
		C \int_{0}^{2R_0}
		[V (\rho) \phi_o(\rho,\tau)
		+ h(\rho,\tau) ] \mathcal Z(\rho)\rho d\rho,\quad
		C =  -(\int_{B_{2} } \eta(\rho) \mathcal Z^2(\rho) \rho d\rho)^{-1}.
	\end{equation*}

	We reformulate \eqref{phio-eq} and \eqref{phii-eq} into the following operators	\begin{equation}\label{z6}
		\begin{aligned}
			\phi_{o}(y,\tau) = \ &
			\TT_{o} [
			J[\phi_{o},\phi_{i}]  ] ,
			\quad
			\phi_{i}(y,\tau) =
			\TT_{2i} \left[
			V \phi_o
			+ h
			+
			c(\tau) \eta(\rho) \mathcal Z(\rho) \right] ,
		\end{aligned}
	\end{equation}
	where $\TT_{o}$ is a linear mapping given by the standard parabolic theory, and $\TT_{2i}$ is given by Proposition \ref{linear-theory}.
	We now solve the system \eqref{z6} by the contraction mapping theorem.
The leading part of the RHS in \eqref{phii-eq} is
 $$H_{1} := h+  C \eta(\rho) \mathcal Z(\rho) \int_0^{2R_0}
	h(\rho,\tau) \mathcal Z(\rho) \rho d\rho .$$
Clearly, $\| H_{1} \|_{v,\ell} \lesssim \|h\|_{v,\ell}$. If $H_1$ satisfies the orthogonality condition in $\mathcal{D}_{2R_0}$, then Proposition \ref{linear-theory} gives the a priori estimate
	\begin{equation*}
		\langle \rho \rangle
		|\pp_\rho \TT_{2i}[H_1]|
		+	|\TT_{2i}[H_1] | \le
		D_i   w_{i}(\rho,\tau)
		\end{equation*}
	provided $\tau^d\gg R_0^6$,
	where  $D_{i}\ge 1$ is a constant and
	\begin{equation*}
		w_{i}(\rho,\tau) =
		v(\tau)\| h \|_{v,\ell} \left(
		R_0^{5-\ell} \ln R_0 \langle \rho\rangle^{-3}+v(\tau) R_0^{1-\ell} (\ln R_0)^2 \langle \rho\rangle^{-1}
		 \right)
			\end{equation*}
So we will choose the space for the inner solution as
	\begin{equation*}
		\B_{i} =
		\left\{
		g(\rho,\tau) \ : \
		\langle \rho \rangle |\pp_\rho g(\rho,\tau)|  + |g(\rho,\tau)| \le
		2D_i  w_{i}(\rho,\tau)
		\right\} .
	\end{equation*}
	For any $\tilde{\phi}_{i}\in \B_{i}$, we will find a solution $\phi_{o} = \phi_{o}[\tilde{\phi}_{i}]$ of \eqref{phio-eq} by the fixed point argument. Let us estimate $J[0,\tilde{\phi}_{i}]$ term by term
	\begin{equation*}
		|A[\tilde{\phi}_{i}] |
		\lesssim
		D_i v
		R_{0}^{-\epsilon_0}
		\langle y \rangle^{-\ell_1}
		\| h \|_{v,\ell}
	\end{equation*}
for some $\epsilon_0>0$ and $\ell_1<\ell$. Also we have
	\begin{equation*}
		| h (1-\eta_{R_0})| \lesssim
		v
		R_{0}^{-\epsilon_0}
		\langle y \rangle^{-\ell_1}
		\| h \|_{v,\ell}
		.
	\end{equation*}
	Consider \eqref{phio-eq} with the right hand side $J[0,\tilde{\phi}_{i}]$. Using $C v \rho(-\Delta_{\R^4})^{-1}[\langle \rho\rangle^{-\ell_1-1}] R_{0}^{-\epsilon_0}
	\| h \|_{v,\ell}$ as the barrier function with a large constant $C$ and then scaling argument, we have
	\begin{equation*}
	\langle \rho\rangle |\pp_\rho \TT_o[J[0,\tilde{\phi}_{i}]](\rho,\tau)| +	|\TT_o[J[0,\tilde{\phi}_{i}]] (\rho,\tau) |
		\le
		w_{o}(\rho,\tau) =
		D_o	D_i v
		R_{0}^{-\epsilon_0}
		\langle \rho \rangle^{2-\ell_1}
		\| h \|_{v,\ell}
	\end{equation*}
	with a large constant $D_o\ge 1$.
	This suggests that we solve $\phi_o$ in the following space:
	\begin{equation*}
		\B_{o} =
		\left\{
		f(\rho,\tau) \ : \
		\langle \rho\rangle |\pp_\rho f(\rho,\tau)|  + |f(\rho,\tau)| \le
		2w_{o}(\rho,\tau)
		\right\} .
	\end{equation*}
	For any $\tilde{\phi}_{o} \in \B_{o}$, since $\rho \le 2 R(\tau)$, we have
	\begin{equation*}
		|V\tilde{\phi}_o(
		1-\eta_{R_0}
		)| \lesssim
		R_{0}^{-2} D_o	D_i v
		R_{0}^{-\epsilon_0}
		\langle \rho \rangle^{-\ell_1}
		\| h \|_{v,\ell} ,
	\end{equation*}
	\begin{equation*}
		| f_1\phi_o + f_2 \rho\pp_\rho \phi_o| \lesssim
		\tau^{-d} R^{2}(\tau)
		D_o	D_i v
		R_{0}^{-\epsilon_0}
		\langle \rho \rangle^{-\ell_1}
		\| h \|_{v,\ell}.
	\end{equation*}
Since $\tau^{-d} R^{2}$, $R_{0}^{-2}\ll 1$, by comparison principle, we have
	\begin{equation*}
		\TT_{o}[
		J[\tilde{\phi}_{o},\tilde{\phi}_{i}] ]  \in \B_{o} ,
	\end{equation*}
and the mapping is a contraction.
	
	Now we have found a solution $\phi_{o} = \phi_{o}[\tilde{\phi}_{i}]\in \B_{o}$. It follows that
	\begin{equation*}
		\left\|
		V \phi_{o}[\tilde{\phi}_{i}]
		+
		C \left[\int_{0}^{2R_0}
		V(\rho) \phi_{o}[\tilde{\phi}_{i}] (\rho,\tau) \mathcal Z(\rho) \rho d\rho\right] \eta(\rho) \mathcal Z(\rho)
		\right\|_{v,\ell}
		\lesssim
		D_o D_i
		R_0^{-\epsilon_0} \| h\|_{v,\ell}.
	\end{equation*}
	Thanks to the choice of $c(\tau)$, $H_2 := V(\rho) \phi_{o}[\tilde{\phi}_{i}]
	+ h
	+
	c[\phi_{o}[\tilde{\phi}_{i}]](\tau) \eta(\rho) \mathcal Z(\rho)$ satisfies the orthogonality condition in $\DD_{2R_0}$.
	By Proposition \ref{linear-theory}, we get
	\begin{equation*}
		\TT_{2i}[h_2] \in \B_{i}
	\end{equation*}
since $R_0^{-\epsilon_{0}}\ll 1$, and similarly it is a contraction mapping. Thus we find a solution
	\begin{equation}\label{phii-upp}
\phi_{i}=\phi_{i}[h] \in \B_{i} ,
\end{equation}
and we obtain a solution
	$(\phi_{o},\phi_{i})$ for \eqref{phio-eq} and \eqref{phii-eq} in the chosen spaces.
		
	Since $\phi_{o}[h] \in \B_{o}$, one has
	\begin{equation*}
		c[h](\tau)
		=
		C \int_{0}^{2R_0}
		h(\rho,\tau)  \mathcal Z(\rho)\rho d\rho
		+
		R_0^{-\epsilon_{0}}  O( v) \|h\|_{v,\ell}
		.
	\end{equation*}
We also have
	\begin{equation*}
		|J[0,\phi_{i}]|  \lesssim
		R_0
		v\langle \rho\rangle^{- \ell } \| h\|_{v,\ell} .
	\end{equation*}
	Using comparison principle to \eqref{phio-eq} repeatedly, we have a refined bound	\begin{equation}\label{phio-upp}
		|\phi_{o}| \lesssim
		R_0
		v\langle \rho\rangle^{2-\ell } \| h\|_{v,\ell} .
	\end{equation}
	Combining \eqref{phii-upp}, \eqref{phio-upp} and then using scaling argument, we conclude
	\begin{equation*}
	\langle \rho\rangle |\pp_\rho \phi| + |\phi|
		\lesssim  R_0^{6-\ell}\ln R_0 v
		\langle \rho \rangle^{-a} \|h\|_{v,\ell}
	\end{equation*}
with $a<\ell -2$.
\end{proof}

\medskip

\section*{Acknowledgements}

\medskip

J. Wei is partially supported by NSERC of Canada.

\medskip

%%%%%%%%%%%%%%%%%%%%
%\newpage

%\bibliography{localbib}
%\bibliography{mrabbrev,localbib}

%\bibliographystyle{amsalpha-lmp}

%\providecommand{\MRhref}[2]{%
 % \href{http://www.ams.org/mathscinet-getitem?mr=#1}{#2}
%}
%\providecommand{\href}[2]{#2}

%\bibliography{RefDatabase}{}
%\bibliographystyle{plain}

\end{document}